\documentclass[12pt]{amsart}
\usepackage{amsthm}
\usepackage{amsfonts, ulem}
\usepackage{amsmath}
\usepackage{amssymb}
\usepackage{tikz}
\usepackage{multicol}
\usepackage[utf8]{inputenc}
\usetikzlibrary{matrix}
\usepackage{enumitem}
\usepackage{color}
\usepackage[all]{xy}

\everymath{\displaystyle}
\CompileMatrices

\CompileMatrices

\newtheorem{thm}{Theorem}[section]
\newtheorem{lemma}[thm]{Lemma}
\newtheorem{prop}[thm]{Proposition}
\newtheorem{cor}[thm]{Corollary}

\theoremstyle{definition}
\newtheorem{defn}[thm]{Definition}
\newtheorem{rmk}[thm]{Remark}

\newcommand{\A}{\mathcal{A}}

\def\R{{\mathbb R}}

\def\C{{\mathbb C}}

\def\N{{\mathbb N}}
\def\T{{\mathbb T}}
\def\Z{{\mathbb Z}}

\def\so{_{\scriptscriptstyle O}}
\def\su{_{\scriptscriptstyle U}}

\def\pr{^{\scriptscriptstyle \R}}
\def\po{^{\scriptscriptstyle O}}
\def\pu{^{\scriptscriptstyle U}}

\def\sr{_{\scriptscriptstyle \R}}
\def\sc{_{\scriptscriptstyle \C}}
\def\crr{^{\scriptscriptstyle {\it CR}}}
\def\crt{^{\scriptscriptstyle {\it CRT}}}

\def\CRT{\mathcal{CRT}}
\def\CR{\mathcal{CR}}
\def\snf{\rm{SNF}}

\newcommand{\fj}{\mathfrak j}
\newcommand{\fJ}{\mathfrak J}
\newcommand{\fB}{\mathfrak B}

\newcommand{\en}{\operatorname{End}}

\newcommand{\sm}[4]
	{ \left( \begin{smallmatrix} {#1} & {#2} \\ {#3} & {#4} \end{smallmatrix} \right)  }
\newcommand{\smv}[2]
	{ \left( \begin{smallmatrix} {#1}  \\ {#2} \end{smallmatrix} \right)  }
\newcommand{\smh}[2]
	{ \left( \begin{smallmatrix} {#1} & {#2}  \end{smallmatrix} \right)  }

\def\id{\text {id} \,}
\def\im{\text {im} \,}
\def\coker{\text {coker} \,}

\newcommand\calg{$C\sp *$-algebra}
\newcommand\ctalg{$C\sp {*, \tau}$-algebra}

\begin{document}

\vspace{-1cm}

\title{$K$-theory for real $k$-graph $C\sp*$-algebras} 
\author{Jeffrey L. Boersema and Elizabeth Gillaspy}

\begin{abstract}
We initiate the study of real $C\sp*$-algebras associated to higher-rank graphs $\Lambda$, with a focus on their $K$-theory.  Following Kasparov and Evans, we identify a spectral sequence which computes the $\CR$ $K$-theory of $C^*\sr(\Lambda, \gamma)$ for any involution $\gamma$ on $\Lambda$, and show that the $E^2$ page of this spectral sequence can be straightforwardly computed from the combinatorial data of the $k$-graph $\Lambda$ and the involution $\gamma$. We provide a complete description of $K\crr(C\sp*\sr(\Lambda, \gamma))$ for several examples of higher-rank graphs $\Lambda$ with involution.
\end{abstract}

\maketitle


\section{Introduction}

Using  the classification of simple purely infinite real $C\sp*$-algebras  \cite{boersema-JFA, brs}, the first author together with Ruiz and Stacey established in \cite[Theorem 11.1]{brs} that for odd $n$, there are two distinct real $C\sp*$-algebras ($\mathcal E_n$ and $\mathcal O_n\pr$) whose complexification is the Cuntz algebra $\mathcal O_n$.  While $\mathcal O_n\pr$ is easy to describe in terms of generators and relations, 
 the only facts known about $\mathcal E_n$ (beyond its existence) are its $K$-theory \cite[Theorem 11.1]{brs} and that it cannot arise as the real  $C\sp*$-algebra of any directed graph \cite[Theorem 6.1]{boersema-MJM}.  This latter fact is quite surprising, since $\mathcal O_n$ is one of the most straightforward examples of a graph $C\sp*$-algebra, and every directed graph gives rise to many potentially different real $C\sp*$-algebras. Indeed \cite{boersema-MJM}, any idempotent graph automorphism $\gamma$ on a graph $E$ gives rise to the real $C\sp*$-algebra $C\sp*\sr(E, \gamma)$ (see Equation \eqref{eq:real-C*-alg} below).

To date, much of the literature on real $C\sp*$-algebras has focused on their $K$-theory (cf.~\cite{schroderbook,boersema2002,brs,boersema-loring}), with some attention paid to other structural properties (cf.~\cite{boersema-kaplansky, rosenberg-real, stacey2003, boersema-stacey2005, boersema-ruiz2011}). In some sense, the $K$-theoretic data is enough:  \cite{boersema-JFA} explains how to construct a purely infinite simple real $C\sp*$-algebra with any appropriate specified $\CR$ $K$-theory, and the $\CR$ $K$-theory is known to be a classifying invariant for simple purely infinite real $C\sp*$-algebras \cite{brs}.  However, the construction in \cite{boersema-JFA} is quite layered and obtuse -- it allows us to detect the existence of real structures for given complex $C\sp*$-algebras, but does not otherwise shine a lot of light. We therefore wish to develop alternative constructions to help us generate more examples of real $C\sp*$-algebras in a concrete way. Specifically, as a test piece, we wish to find a concrete representation of the real $C\sp*$-algebras $\mathcal E_n$.

{To this end, we introduce} in this paper the real $C\sp*$-algebra $C\sp*\sr(\Lambda, \gamma)$ associated to a higher-rank graph $\Lambda$ and an involution $\gamma$ on $\Lambda$.  
Inspired by Robertson and Steger \cite{robertson-steger}, Kumjian and Pask introduced higher-rank graphs, or $k$-graphs,  in \cite{kp}, as a way to construct combinatorial examples of $C\sp*$-algebras which are more general than graph $C\sp*$-algebras.  
In addition to their intrinsic links with a variety of combinatorial structures, such as buildings \cite{robertson-steger,konter-vdovina} 
and ultrametric Cantor sets \cite{FGJKP-wavelets, FGLP, kang-etc-ultrametrics},  (complex) $k$-graph $C\sp*$-algebras have provided important examples for   Elliott's classification program \cite{ruiz-sims-sorensen} 
as well as for noncommutative geometry \cite{pask-rennie-sims, FGJKP-wavelets}.

The family of real $C\sp*$-algebras that arise from higher-rank graphs with involution is much larger than the family arising from graph algebras. This follows, for example, from the fact that the  $K_1$-group of a (complex) graph $C^*$-algebra must be torsion free \cite{raeburn-szyman}, a restriction which disappears for higher-rank graphs.
However, in order to answer the question of whether the exotic Cuntz algebra $\mathcal E_n$ arises from a higher-rank graph, we need to be able to compute the $K$-theory of real higher-rank graph $C\sp*$-algebras, 
since $\mathcal E_n$ can only be identified by its $K$-theory. 
In this article we will develop the methods to carry out the computations of the $K$-theory of such algebras and will demonstrate these methods with several detailed computations.

While the $K$-theory of a graph $C\sp*$-algebra can be computed from the adjacency matrix of the graph using a long exact sequence (see \cite{bhrs} and \cite{raeburn-szyman}), the situation is more complicated for a  higher-rank graph $\Lambda$.  In \cite{evans}, Evans identified a spectral sequence which converges to the $K$-theory of $C\sp*(\Lambda)$, and computed the $K$-theory explicitly in some low-rank situations.  
Thus, we will first confirm that given the real $C\sp*$-algebra of a higher-rank graph, there exists a spectral sequence which converges to its $K$-theory.  This is the focus of Section \ref{sec:spectral-sequence}. For a real $C\sp*$-algebra $A$, the $K$-theoretic invariant that we consider contains much more information than does the $K$-theory of a complex $C\sp*$-algebra. We will consider the so-called $ \CR$ $K$-theory of $A$, which includes not only the 8 real $K$-groups $K_*(A)$, but also the two $K$-theory groups  $K_*(A_\C)$ of its complexification, as well as  a number of homomorphisms between the various groups that satisfy certain compatibility conditions. (See Section \ref{sec:CR} below.)  Fortunately, Theorem \ref{sp-seq1} confirms that the Evans spectral sequence is sufficiently functorial that it will indeed contain all of this additional structure.

Thus, the spectral sequence that we develop is simultaneously a generalization of the spectral sequence of Evans (for a complex higher rank graph algebra) and the long exact sequence developed in \cite{boersema-MJM} (for a real algebra from a graph with involution). 
Similar to the long exact sequence found in \cite{boersema-MJM}, the building blocks of our spectral sequence consist of direct sums of the $K$-theory of $\C$ and $\R$, viewed as real $C\sp*$-algebras. {Indeed, we show in Section \ref{sec:lowrank}  that the $E^2$ page of the spectral sequence can be computed from a chain complex, whose {entries} are the aforementioned direct sums of the $K$-theory of $\C$ and $\R$, and whose boundary maps} are determined by the combinatorial structure of the higher-rank graph. 

When it comes to computing the $\CR$ $K$-theory of specific examples of real $C\sp*$-algebras, the complicated structure of real $K$-theory is both boon and bane.  While the intricacy of $\CR$ $K$-theory adds many additional steps to certain computations, the circumscribed relationships between the various groups (described in Section \ref{sec:CR}) mean that often, the entire $\CR$ $K$-theory is completely determined by just a few of its constituent groups and homomorphisms.  Consequently, as we see in Section \ref{sec:examples}, a small amount of information frequently enables us to completely describe the $\CR$ $K$-theory.  
To be precise, in Section \ref{sec:examples}, we use both the simplified description of the $E^2$ page of the spectral sequence from Section \ref{sec:lowrank}, and the relationships between the structure maps of $\CR$ $K$-theory, to completely describe the $\CR$ $K$-theory of several examples of rank-2 graphs. In particular, for each odd $n$, we identify in Section \ref{SectionExample2} a 2-graph $\Lambda$ and an involution $\gamma$ on $\Lambda$ such that $C\sp*(\Lambda) $ is $KK$-equivalent to $\mathcal O_n \otimes \mathcal O_n$, but {its real structure} $C\sp*_\R(\Lambda, \gamma)$ is not a tensor product of real Cuntz algebras. In other words, we have discovered {new} real structures on $\mathcal O_n \otimes \mathcal O_n$, other than $\mathcal O_n\pr \otimes\sr \mathcal O_n\pr$. 

{However, we have not yet discovered an example of a higher-rank graph with involution whose associated $C\sp*$-algebra is $\mathcal E_n$. We discuss this, and other open questions, in Section \ref{sec:questions}.}

\section{Preliminaries}
\label{sec:background}
\subsection{Higher-rank graphs and their (real) $C\sp*$-algebras}
\label{sec:kgraphs}
Higher-rank graphs were introduced by Kumjian and Pask in \cite{kp}.  To define them, we first specify that throughout this paper, we view $\N^k$ as a category with one object (namely $0$), where composition of morphisms is given by addition.

For consistency with the usual notation $n \in \N^k$ to describe a $k$-tuple of natural numbers (which, in the category-theoretic perspective, is a morphism in $\N^k$), we will write $\lambda \in \Lambda$ to denote a morphism in the category $\Lambda$.  We will identify a category's objects with the identity morphisms, so that statements such as $0 \in \N^k$ are still allowed.
\begin{defn}
A {\em higher-rank graph} of rank $k$, or a {\em $k$-graph}, is a countable small category $\Lambda$  equipped with a {\em degree functor} $d: \Lambda \to \N^k$ such that, whenever a morphism $\lambda \in \Lambda$ satisfies $d(\lambda) = m + n$, there exist unique morphisms $\mu, \nu \in \Lambda$ such that $\lambda = \mu \nu, d(\mu) = m, d(\nu) = n$. 
\label{def:k-graph}
\end{defn}

Equivalently (cf.~\cite[Theorems 4.3 and 4.4]{hazle-raeburn-sims-webster}, \cite[Theorem 2.3]{EFGGGP}) a $k$-graph consists of a directed graph $G$, with $k$ colors of edges, and a {\em factorization rule} $\sim$ on the multicolored paths.  For each pair of colors (``red'' and ``blue'' for this discussion), and each pair of vertices $v, w$, the factorization rule identifies each path from $v$ to $w$ which consists of a blue edge followed by a red edge (a {\em blue-red path}) with a unique red-blue path from $v$ to $w$.  The factorization rule must also satisfy certain consistency conditions which ensure that, for each path in the graph, its equivalence class  under $\sim$ corresponds to a $k$-dimensional hyper-rectangle.  

Notice that the factorization rule preserves the number of edges of each color in a given hyper-rectangle.  Thus, writing $G\sp*$ for the path category of $G$, and identifying the set  of  $k$ colors with the standard generators $e_1, \ldots, e_k$ of $\N^k$, we obtain a well-defined morphism $d: G\sp*/\sim \to \N^k$.  With this notation, $(G\sp*/\sim, d)$ is a $k$-graph as in Definition \ref{def:k-graph}.

Let $\Lambda$ be a $k$-graph.  Given objects $v, w \in \Lambda$ and   $n \in \N^k$, we write 
\begin{equation}
\Lambda^n = \{ \lambda \in \Lambda: d(\lambda) = n\}, \qquad v\Lambda = \{ \lambda \in \Lambda: r(\lambda) = v\}, \qquad \Lambda^n w = \{ \lambda \in \Lambda: s(\lambda) = w \text{ and } d(\lambda) = n\},
\label{eq:k-graph notation}
\end{equation}
as well as the obvious variations.  Observe that $\Lambda^0$ is the set of objects of $\Lambda$, which we also denote as {\em vertices} thanks to the graph-theoretic inspiration for $k$-graphs.  We say that $\Lambda$ is {\em row-finite} if $|v\Lambda^n | < \infty$ for all $n\in \N^k$ and $v \in \Lambda^0$, and that $\Lambda$ is {\em source-free} if, for all $n$ and $v$, $v\Lambda^n \not= \emptyset$.\footnote{Equivalently, $\Lambda$ is source-free if $v\Lambda^{e_i}$ is nonempty for all $1 \leq i \leq k$.}

\begin{defn}\cite{kp}
Given a row-finite source-free $k$-graph $\Lambda$, a {\em Cuntz--Krieger $\Lambda$-family} is a collection $\{t_\lambda\}_{\lambda \in \Lambda}$ of partial isometries  in a $C\sp*$-algebra $A$ which satisfy the following conditions:
\begin{enumerate}
\item[(CK1)] For each $v \in \Lambda^0$, $t_v$ is a projection, and $t_v t_w = \delta_{v,w} t_v$.
\item[(CK2)] For each $\lambda \in \Lambda, \ t_\lambda^* t_\lambda = t_{s(\lambda)}$.
\item[(CK3)] For each $\lambda, \mu \in \Lambda, \ t_\lambda t_\mu = t_{\lambda \mu}$.
\item[(CK4)] For each $v \in \Lambda^0$ and each $n \in \N^k$, 
\[ t_v = \sum_{\lambda \in v\Lambda^n} t_\lambda t_\lambda^*.\]
\end{enumerate}
We define $C\sp*(\Lambda)$ to be the universal (complex) $C\sp*$-algebra generated by a Cuntz--Krieger family, in the sense that for any Cuntz--Krieger $\Lambda$-family $\{t_\lambda\}_{\lambda \in \Lambda}$, there is a surjective $*$-homomorphism $C\sp*(\Lambda) \to C\sp*(\{ t_\lambda\}_\lambda)$.  
\end{defn}
We write $\{s_\lambda\}_{\lambda \in \Lambda}$ for the generators of $C\sp*(\Lambda)$.
One computes easily, using the Cuntz--Krieger relations, that $C\sp*(\Lambda) = \overline{\text{span}} \{ s_\lambda s_\mu^*: s(\lambda) = s(\mu)\}.$

Given a $k$-graph $\Lambda$, we now describe how to associate a real $C\sp*$-algebra to it. We assume that $\Lambda$ is row-finite and has no sources.
Observe that there is a (unique) anti-multiplicative linear automorphism $\chi$ of $C \sp * (\Lambda)$ which satisfies $\chi(s_\lambda) =  s_\lambda^*$.

\begin{defn}
An {\em involution} $ \gamma$ on a $k$-graph $\Lambda$ is a degree-preserving functor $\gamma \colon \Lambda \rightarrow \Lambda$ which satisfies 
$ \gamma \circ  \gamma = \id_{\Lambda}$. 
\label{def:involution}
\end{defn}
The functoriality of $ \gamma$ implies that $s   \gamma =  \gamma  s$ and $ r   \gamma =  \gamma  r$ for any involution $ \gamma$.

Given an involution $ \gamma$ on $\Lambda$, the elements $\{s_{ \gamma(\lambda)}: \lambda \in \Lambda\}$ form a Cuntz--Krieger $\Lambda$-family, so the universal property of $C \sp*(\Lambda)$ implies the existence of an automorphism $C\sp*(\gamma)$ on $C \sp*(\Lambda)$, given by 
$C\sp*(\gamma)(s_\lambda) := s_{ \gamma(\lambda)}.$
 Since $\chi$ commutes with $C\sp*(\gamma)$, the composition $\widetilde \gamma := \chi \circ C\sp*( \gamma)$ is an antimultiplicative involution of $C \sp *(\Lambda)$, which is determined uniquely by the formula 
 \[ \widetilde \gamma(s_\lambda) = s^*_{ \gamma(\lambda)}.\]

It follows that $(C \sp *(\Lambda); \widetilde \gamma)$ is a \ctalg 
~(this just means exactly that $\widetilde \gamma$ is an antiautomorphism of $C \sp *(\Lambda)$). The corresponding real \calg ~is given by 
\begin{equation}
\label{eq:real-C*-alg}
C\sr \sp*(\Lambda, \gamma): = \{a \in C \sp *(\Lambda) \mid \widetilde \gamma(a) = a^* \} \; 
\end{equation}
(see Definition~1.1.4 of \cite{schroderbook} and the following remark).

\begin{lemma}
\label{lem:real-alg-elts}
Given an involution $ \gamma$ on a row-finite source-free $k$-graph $\Lambda$, 
\[C\sr^*(\Lambda, \gamma) = \overline{\rm{span}}_\R \{ z s_\lambda s_\mu^* + \overline{z} s_{ \gamma(\lambda)} s_{ \gamma(\mu)}^* \mid z \in \C, \lambda, \mu \in \Lambda\}.\]
\end{lemma}
\begin{proof}
Define $A$ to be the right hand side.
We first observe that 
\[ (z s_\lambda s_\mu^* + \overline{z} s_{ \gamma(\lambda)} s_{ \gamma(\mu)}^*)^* = \overline z s_\mu s_\lambda^*  + z s_{ \gamma(\mu)} s_{ \gamma(\lambda)}^*,\]
whereas the fact that $\widetilde \gamma$ is antimultiplicative but linear  implies that  we also have 
\[ \widetilde \gamma(z s_\lambda s_\mu^* + \overline{z} s_{ \gamma(\lambda)} s_{ \gamma(\mu)}^*) = z s_{ \gamma(\mu)} s_{ \gamma(\lambda)}^* + \overline{z} s_{\mu} s_{\lambda}^*.\]
Hence $A \subseteq C\sr^*(\Lambda, \gamma)$.

To see that $A \cong C\sr^*(\Lambda, \gamma)$, we will show that $A + i A = C\sp*(\Lambda)$.  To that end, fix $\alpha \in \C$ and consider $\alpha s_\lambda s_\mu^* \in C\sp*(\Lambda)$.  If we set $z = \alpha/2, w = - i \alpha/2$,  a quick computation reveals that 
\[ z s_\lambda s_\mu^* + \overline{z} s_{ \gamma(\lambda)} s_{ \gamma(\mu)}^* + i ( w s_\lambda s_\mu^* + \overline{w} s_{ \gamma(\lambda)} s_{ \gamma(\mu)}^*) = \alpha s_\lambda s_\mu^*.\]
As the elements $\alpha s_\lambda s_\mu^*$ densely span $C\sp*(\Lambda)$ as a real vector space, we conclude that 
$A + i A = C\sr^*(\Lambda,\gamma)$ as claimed.
\end{proof}

To each $k$-graph we can associate $k$ commuting matrices $M_1, \ldots, M_k$ in $M_{\Lambda^0}(\N)$:
\begin{equation}
\label{eq:matrices}
M_i(v, w) := | v\Lambda^{e_i} w|,
\end{equation}
that is, the $(v, w)$ entry in $M_i$ counts the number of color-$i$ edges in $\Lambda$ with source $w$ and range $v$.  We call the matrices $M_i$ the {\em incidence matrices} or {\em adjacency matrices} of the $k$-graph. The fact that $M_i M_j = M_j M_i$ follows from the requirement, imposed by the factorization rule, that there be an identical number of blue-red as red-blue paths between any given pair $(v, w)$ of vertices.

Given a $k$-graph $\Lambda$, we can form the {\em skew product} $\Lambda \times_d \Z^k$, with $\text{Obj}(\Lambda \times_d \Z^k) = \Lambda^0 \times \Z^k$ and $\text{Mor}(\Lambda \times_d \Z^k) = \Lambda \times \Z^k$.  We have $s(\lambda, n) = (s(\lambda), n + d(\lambda))$ and $r(\lambda, n) = (r(\lambda), n)$.  Defining $d: \Lambda \times_d \Z^k \to \N^k$ by $d(\lambda, n) = d(\lambda)$ makes $\Lambda \times_d \Z^k$ a $k$-graph, which is row-finite and source-free whenever $\Lambda$ is.  Moreover, \cite[Theorem 5.7]{kp} the universal property of $C\sp*(\Lambda \times_d \Z^k)$ implies that $C\sp*(\Lambda \times_d \Z^k)$ 
admits an action of $\Z^k$, given on the generators by $s_{\lambda, n} \cdot m := s_{\lambda, m+n}$.

\subsection{$K$-theory for real $C\sp*$-algebras}
\label{sec:CR}
The main $K$-theoretic invariant that we will use in the category of real $C\sp*$-algebras is $\CR$ $K$-theory, denoted by $K\crr(A)$ for a real $C\sp*$-algebra $A$. In this section, we will review the definition and the properties of this invariant, necessary for the rest of this article. The components of this invariant are described in Section 1.4 of \cite{schroderbook}. It is a somewhat smaller invariant than ``united $K$-theory,'' denoted $K\crt(A)$, described in Section 1 of \cite{boersema2002} and utilized in \cite{boersema2004} and \cite{brs}; 
but by \cite[Theorem 4.2.1]{hewitt} it contains equivalent information.

For a real $C\sp*$-algebra $A$, we define $K\crr(A) = \{ KO_*(A), KU_*(A) \}$ where $KO_*(A)$ is the usual period-8 $K$-theory of $A$ and $KU_*(A) := K_*(\C \otimes A)$ is the $K$-theory of the complexification of $A$. In addition, $KO_*(A)$ has the structure of a graded module over the ring $KO_*(\R)$ where the groups of this ring are given by
\[ KO_*(\R) = \Z \quad \Z_2 \quad \Z_2 \quad 0  \quad \Z \quad 0 \quad 0 \quad 0 \]
in degrees 0 through 7. 

In particular, 
multiplication by the non-trivial element $\eta$ of $KO_1(\R) \cong \Z_2$
 induces a natural transformation
$\eta \colon KO_i(A) \rightarrow KO_{i+1}(A) .$
We note that $\eta$ satisfies the relations $2 \eta = 0$ and $\eta^3 = 0$, both as an element of the ring $KO_*(\R)$ and as a natural transformation. There is also a non-trivial element $\xi \in KO_4(\R)$, and corresponding natural transformation, that satisfies $\xi^2 = 4 \beta\so$ where $\beta\so$ is the real Bott periodicity isomorphism of degree 8.

Complex $K$-theory $KU_*(A)$ has the structure of a module over $KU_*(\R) = K_*(\C)$, but the only natural transformation which arises from this structure is the degree-2 Bott periodicity map $\beta$. There is, however, a natural transformation $\psi \colon KU_*(A) \rightarrow KU_*(A)$ that arises from the conjugation map $\psi \colon \C \otimes A \rightarrow \C \otimes A$ defined by $a + ib \mapsto a - ib$.

In addition, there are natural transformations
\begin{align*}
c \colon & KO_*(A) \rightarrow KU_*(A) \\
r \colon & KU_*(A) \rightarrow KO_*(A) 
\end{align*}
which are induced by the natural inclusion maps $\R \hookrightarrow \C$ and $\C \hookrightarrow M_2(\R)$, respectively. 

Taken together, these natural transformations satisfy the following set of relations:
\begin{align}
\label{eq:natural-transformations}
rc &= 2  & cr &= 1 + \psi &2 \eta &= 0  \notag \\
\eta r &= 0 & c \eta&= 0 &  \eta^3 &= 0    \notag \\
r \psi &= r & \psi^2 &= \id &  \xi^2 &= 4 \beta\so \\ 
\psi c &= c & \psi \beta &= -\beta\su \psi & \xi &= r \beta^2 c \notag 
\end{align} 

A pair $(G^O, G^U)$ of $\Z$-graded abelian groups ($G^O$ with period 8 and $G^U$ with period 2) together with natural transformations $\eta, \beta, \zeta,\psi, r, c$ as above, such that the equations \eqref{eq:natural-transformations} hold, is called a {\em $\CR$-module}, and the category of such objects is the target of the functor $K\crr(A)$. 

We display the full structure of $K\crr(\R)$ and $K\crr(\C)$ in Tables \ref{crtR} and \ref{crtC} below. These are the only two singly-generated free $\CR$-modules (up to suspensions) and all of the relations above are encoded in these two $\CR$-modules. Furthermore, these $\CR$-modules will be the building blocks of the spectral sequence we will use to compute $K\crr( C\sp*\sr(\Lambda, \gamma))$.  (See Theorem \ref{sp-seq2} and Section \ref{sec:lowrank} below.)

\begin{table}[h]
\caption{$K\crr(\R)$} \label{crtR}
\[ \begin{array}{|c|c|c|c|c|c|c|c|c|c|}  
\hline \hline  
n & ~0~ & ~1~ & ~2~ & ~3~ & ~4~ & ~5~ & ~6~ & ~7~  \\
\hline  \hline
KO_n
& \Z  & \Z_2  & \Z_2  & 0 
	& \Z  & 0 & 0 & 0  \\
\hline  
KU_n 
& \Z  & 0 & \Z  & 0 & \Z   & 0 
	& \Z  & 0   \\
\hline \hline
\eta_n & 1 & 1 & 0 & 0 & 0 & 0 & 0 & 0 \\
\hline 
c_n & 1 & 0 & 0 & 0 & 2 & 0 & 0 & 0   \\
\hline
r_n & 2 & 0 & 1 & 0 & 1 & 0 & 0 & 0      \\
\hline
\psi_n & 1 & 0 & -1 & 0 & 1 & 0 & -1 & 0    \\
\hline \hline
\end{array} \]
\end{table}

\begin{table} 
\caption{$K\crr(\C)$}  \label{crtC}
\[ \begin{array}{|c|c|c|c|c|c|c|c|c|c|}  
\hline  \hline 
n & \makebox[1cm][c]{0} & \makebox[1cm][c]{1} & 
\makebox[1cm][c]{2} & \makebox[1cm][c]{3} 
& \makebox[1cm][c]{4} & \makebox[1cm][c]{5} 
& \makebox[1cm][c]{6} & \makebox[1cm][c]{7}  \\
\hline  \hline
KO_n 
& \Z & 0 & \Z & 0 & \Z & 0 & \Z & 0   \\
\hline  
KU_n 
& \Z \oplus \Z & 0 & \Z \oplus \Z & 0 & \Z \oplus \Z & 0 
	& \Z \oplus \Z  & 0   \\
\hline \hline
\eta_n & 0 & 0 & 0 & 0 & 0 & 0 & 0 & 0 \\
\hline 
c_n & \smv{1}{1} & 0 & \smv{-1}{1} & 0 & \smv{1}{1} & 0 & \smv{-1}{1} & 0 
	   \\
\hline
r_n & \smh{1}{1} & 0 & \smh{-1}{1} & 0 & \smh{1}{1} & 0 
	& \smh{-1}{1} & 0      \\
\hline
\psi_n & \sm{0}{1}{1}{0} & 0 & \sm{0}{-1}{-1}{0} & 0 & 
	\sm{0}{1}{1}{0} & 0 & \sm{0}{-1}{-1}{0} & 0    \\
\hline \hline
\end{array} \]
\end{table}

The natural transformations also combine to form a long exact sequence
\begin{equation}
\label{eq:CR-LES}
\dots \rightarrow
KO_{i}( A) \xrightarrow{\eta} 
KO_{i+1}( A) \xrightarrow{c} 
KU_{i+1}( A) \xrightarrow{r \beta^{-1}} 
KO_{i-1}(A) \rightarrow
\cdots  \; .
\end{equation}
The following theorem summarizes some of the important properties of the invariant $K\crr(A)$. 

\begin{thm} \label{CR Theorems} ~
\begin{enumerate}
\item If $A$ is a real $C\sp*$-algebra, then $K\crr(A)$ is a $\CR$-module.
\item If $A$ and $B$ are real $C\sp*$-algebras, such that $\C \otimes A$ and $\C \otimes B$ are in the bootstrap category $\mathcal{N}$, then
$A$ and $B$ are $KK$-equivalent if and only if $K\crr(A) \cong K\crr(B)$.
\item If $A$ and $B$ are real $C\sp*$-algebras, such that $\C \otimes A$ and $\C \otimes B$ are purely infinite simple Kirchberg algebras, then
$A$ and $B$ are stably isomorphic if and only if $K\crr(A) \cong K\crr(B)$.
\item If $A$ and $B$ are real $C\sp*$-algebras, such that $\C \otimes A$ and $\C \otimes B$ are purely infinite simple unital Kirchberg algebras, then
$A$ and $B$ are isomorphic if and only if $(K\crr(A), [1_A]) \cong (K\crr(B), [1_B])$.
\end{enumerate}
\end{thm}

\begin{proof}
From
\cite[Theorem~1.12]{boersema2002}, we know that $K\crt(A)$ is a $\CRT$-module, from which it follows immediately that $K\crr(A)$ is a $\CR$-module.
By
 \cite[Corollary~4.11]{boersema2004} and \cite[Theorem~10.2]{brs}, 
we know that statements (2), (3), and (4) are true when $K\crr(-)$ is replaced by $K\crt(-)$ throughout.
However, from \cite[Theorem~4.2.1]{hewitt}, we know that $K\crr(A) \cong K\crr(B)$ if and only if $K\crt(A) \cong K\crt(B)$.

\end{proof}

From the point of view of calculations, it is often the case that once $KU_*(A)$ and a few of the $KO_*(A)$ groups are known, 
 then the rest can be identified using the rich structure of a $\CR$-module, specifically the relations among the natural transformations above combined with the long exact sequence.
 In \cite{hewitt}, Beatrice Hewitt found a way to boil down the information from an acyclic $\CR$-module into a simpler structure, called the {\it core} of the $\CR$-module. We will introduce this helpful structure in Section~\ref{sec:examples} and use it to facilitate calculations of the $K$-theory of some specific higher-rank graph algebras.

\section{The spectral sequence}
\label{sec:spectral-sequence}

This section, which is the theoretical cornerstone of the paper, takes inspiration from Evans' computations \cite{evans} of $K$-theory for the complex $C\sp*$-algebras of higher-rank graphs.  Our goal is to obtain a computable description of the spectral sequence which converges to $K\crr(C\sp*\sr(\Lambda, \gamma))$.  The spectral sequence in question was introduced by Kasparov in \cite[6.10]{kasparov-equivKK} and applies to crossed product $C\sp*$-algebras.  Thus, we begin by showing in Theorem \ref{B_times_Z} that $C\sp*\sr(\Lambda, \gamma)$ is stably isomorphic to $C\sp*\sr(\Lambda \times_d \Z^k, \gamma) \rtimes \Z^k$. 
Next, we establish (see Theorem \ref{sp-seq1}) that Kasparov's spectral sequence \cite{kasparov-equivKK} encodes not only the real and complex $K$-theory groups, but also the $\mathcal{CR}$-module structure linking them.  
Having thus established the relevance of Kasparov's spectral sequence to our situation, in Section \ref{sec:combinatorial}
we combine the AF structure of $C\sp*\sr(\Lambda \times_d \Z^k, \gamma) $ (Corollary \ref{cor:AF}) and its $\Z^k$-module structure to provide a more combinatorial description of the $E^2$ page of the spectral sequence in Theorem \ref{sp-seq2}.  Namely, we identify a chain complex whose homology computes  the $E^2$ page of the spectral sequence.  Our approach here follows the outline used by Evans for complex $C\sp*$-algebras in \cite{evans}, although the intricate structure of real $K$-theory necessitates a few detours.

Thanks to the AF structure of $C\sp*\sr(\Lambda \times_d \Z^k, \gamma)$, the building blocks of this chain complex are direct sums of the $K$-theory of the two most basic real $C\sp*$-algebras, namely $\R$ and $\C$.  In this situation, Lemma \ref{CRMagic}  establishes that the entire $\mathcal{CR}$-module structure is dictated by what happens at the level of the complex $K$-theory.  Combining this insight with Evans'  computations of the $K$-theory of complex $k$-graph $C\sp*$-algebras, we provide in  Section \ref{sec:lowrank}  a more explicit description of the $E^2$ page of the spectral sequence for $k$-graphs with $k \leq 3$ and finitely many vertices.  This description is fundamental to our analysis of the examples in Section \ref{sec:examples}.

\subsection{Structure of $C\sr^*(\Lambda, \gamma)$}

Given a $k$-graph $(\Lambda,  \gamma)$ with involution, we can extend $ \gamma$  to an involution (also denoted $ \gamma$) on the skew-product $k$-graph $\Lambda \times_d \Z^k$ by the formula
\[ {\gamma} (\mu, n) = ( \gamma(\mu), n) \; . \]
The involution $\gamma$ thus induces a real structure on the complex $C\sp*$-algebra $B = C\sp*(\Lambda \times_d \Z^k)$; we will write 
\[B\sr = C\sr^*(\Lambda \times_d \Z^k, \gamma)\]
 for the associated real $C\sp*$-algebra.
Recall that there is an action $\beta$ of $\Z^k$ on $B$, given by $\beta(n) \cdot s_{\mu, m} = s_{\mu, m+n}$. Using the description of $B\sr$ which arises from Lemma~\ref{lem:real-alg-elts}, it is easy to see that the action $\beta$ restricts to an action (also denoted $\beta$) of $\Z^k$ on $B\sr$. We will also use the notation $\beta_i = \beta(e_i)$ for $i \in \{1, 2, \dots, k\}$.

\begin{thm} \label{B_times_Z}
There is an isomorphism 
\[ B\sr \rtimes_\beta \Z^k \cong C\sr^*(\Lambda,  \gamma) \otimes\sr \mathcal{K}\sr \; \]
and hence
\[ K\crr (C\sr^*(\Lambda,  \gamma)) \cong K\crr(B\sr \rtimes_\beta \Z^k) \; .\]
\end{thm}

\begin{proof}
As in Corollary 5.3 of \cite{kp}, there is an isomorphism 
$C\sp*(\Lambda) \rtimes_\alpha \mathbb T^k \cong C\sp*(\Lambda \times_d \Z^k) $ of complex $C\sp*$-algebras,
where $\alpha$ is the gauge action of $\mathbb T^k$ on $C\sp*(\Lambda)$. Furthermore, under this isomorphism, the dual action of $\Z^k$ on $C\sp*(\Lambda) \rtimes_\alpha \T^k$ corresponds to the action $\beta$ on $B = C\sp*(\Lambda \times_d \Z^k)$ described above.
By Takai duality (for complex $C\sp*$-algebras) we then have
\begin{equation}  \label{takai}
B \rtimes_\beta \Z^k \cong C\sp*(\Lambda \times_d \Z^k) \rtimes_\beta \Z^k
	\cong (C\sp*(\Lambda) \rtimes_\alpha \mathbb T^k) \rtimes_\beta \Z^k \cong C\sp*(\Lambda) \otimes \mathcal{K} \; . \end{equation}
So far, all of this is exactly as indicated in \cite{kp}.

Now we consider the involutions on each of these $C\sp*$-algebras and show that the isomorphisms all respect the corresponding involutions. Recall from \cite[Section 2]{boersema-RMJ} that a real $C\sp*$-dynamical system consists of a quintuple $(A, \overline{ \, \cdot \, }, G, \overline{ \, \cdot \,} , \alpha)$ where $(A, \overline{ \, \cdot \,})$ is a complex $C\sp*$-algebra with conjugate-linear involution; $(G, \overline{ \, \cdot \, })$ is a group with involution; and $\alpha$ is an action of $G$ on $A$ intertwining the involutions in the sense that
\[ \overline{\alpha(g) (a)} = \alpha( \overline{g})  (\overline{a}) \quad \text{for all $a \in A$, $g \in G$.} \]
If $(A, \overline{ \, \cdot \, } , G,\overline{ \, \cdot \, }, \alpha)$ is a real $C\sp*$-dynamical system then the crossed product $A \rtimes_\alpha G$ inherits a natural conjugate-linear involution (Theorem~2 of \cite{boersema-RMJ}).

In our case, it is straightforward to check that the involution $\widetilde \gamma$ on $B$ commutes with the action of $\beta$ so that $(B, \widetilde \gamma, \Z^k, \id, \beta)$ is a real $C\sp*$-dynamical system. Similarly, the gauge action $\alpha$ intertwines with the involution $\widetilde \gamma$ on $C\sp*(\Lambda)$ so that $( C\sp*(\Lambda), \widetilde \gamma, \T^k, \tau, \alpha)$ is also a real $C\sp*$-dynamical system.
Here $\tau$ is the involution on $\T^k$ given by $\tau(z_1, \dots, z_k) = (\overline{z_1}, \dots, \overline{z_n})$. 

Furthermore, as groups with involution, $(\T^k, \tau)$ is dual to $(\Z^k, \id)$ in the sense of \cite[Section 3]{boersema-RMJ}. Therefore by Takai duality for real $C\sp*$-algebras (\cite[Theorem 9]{boersema-RMJ}), the isomorphisms of Equation~(\ref{takai}) are isomorphisms that respect the real structures, proving the theorem.
\end{proof}

As in \cite{evans}, for any $m\in \Z^k$ and $v\in \Lambda^0$, let
\begin{align*}
&B_m(v) = \overline{\text{span}} \{ s_{\mu, m- d(\mu)} s_{\nu, m- d(\nu)}^* \mid s(\mu) = s(\nu) =v \}    \\
\text{~and~}
&B_m = \overline{\text{span}} \{ s_{\mu, m- d(\mu)} s_{\nu, m- d(\nu)}^* \mid s(\mu) = s(\nu) = v\text{~for some $v \in \Lambda^0$ } \}    \\
\end{align*}
Then, as in~\cite[Lemma 3.4]{evans} or \cite[Lemma 5.4]{kp}, we have $B =  \lim_{m \to \infty} B_m$ and there 
are isomorphisms
\[
B_m(v)  \cong \mathcal K(\ell^2(s^{-1}(v)))\quad \text{and} \quad
B_m  \cong  \bigoplus_{v \in \Lambda^0} B_m(v) \; , \]
which describe the structure of $B$ as an AF-algebra.
For $m \leq n$, the inclusion map $\iota_{nm} \colon B_m \hookrightarrow B_n$  is determined on $s_{\mu, m - d(\mu)} s_{\nu, m - d(\nu)}^* \in B_m(v)$ by the fact that, by (CK4),
	\begin{equation} s_{\mu, m-d(\mu)} s_{\nu, m-d(\nu)}^* = \sum_{\stackrel{r(\alpha) = v}{d(\alpha) = n-m}} s_{\mu \alpha, m - d(\mu)} s_{\nu\alpha, m -d(\nu)}^* \; .
	\label{eq:iota}
	 \end{equation}
Observe that the terms on the right-hand side all lie in $B_n$, as $d(\mu \alpha) + m - d(\mu) = d(\alpha) + m = n$; however, they will generally lie in different summands $B_n(w)$.
	 	 
Now, we consider the real structure on $B_m$ and $B$. The involution $\widetilde \gamma$ on $B = C\sp*(\Lambda \times_d \Z^k )$ induced by $\gamma$ satisfies 
$\widetilde \gamma(s_{\lambda, m}) = s_{\gamma(\lambda), m}^*$,  so we have $\widetilde \gamma( B_m(v) ) = B_m(\gamma (v))$ and $\widetilde \gamma( B_m ) = B_m$.
Therefore $ \widetilde \gamma$ gives a real structure on $B_m(v)$ (when $v$ is a vertex fixed by $\gamma$) and on $B_m(v) \oplus B_m(\gamma (v))$ (when $v$ is not fixed by $\gamma$).
 The following lemma describes the structure of the corresponding real \calg s $B_m(v)\sr$ and $\left( B_m(v) \oplus B_m(\gamma (v)) \right)\sr$.

\begin{lemma}
	With notation as above, if $ \gamma(v) = v$, then $B_m(v)_\R \cong \mathcal K_\R(\ell^2(s^{-1}(v))).$  
If $  \gamma(v) \neq v$ then $(B_m(v) \oplus B_m(  \gamma(v)))_\R \cong \mathcal K_\C(\ell^2(s^{-1}(v)))$.
\label{lem:AF}
	\end{lemma}
\begin{proof}
We first consider the case when $ \gamma(v) = v$.  
Fix $j \in \N$ and decompose 
\[J = J(j) := \{\lambda \in \Lambda v : d(\lambda) \leq (j, j, \ldots, j)\}\] 
as $J = J_f \sqcup J_1 \sqcup J_2$, where $ \gamma|_{J_f} = \id$ and $ \gamma(J_1) = J_2$.
We can view elements of $M_J(\C)$ as lying in $B_m(v)$ under the identification $e_{\mu, \nu} \mapsto s_{\mu, m-d(\mu)} s_{\nu, m-d(\nu)}^*$. {With this identification, the antimultiplicative involution $\widetilde \gamma$ is given on $M_J(\C)$ by $\widetilde \gamma(e_{\mu,\nu}) = e_{\gamma(\nu), \gamma(\mu)}$. }Furthermore,  $B_m(v) = \varinjlim_{j\to \infty}  M_{J(j)}(\C)$; the connecting map $M_{J(j)} \to M_{J(j+1)}$ is determined  by the inclusions $J_f(j) \subseteq J_f(j+1)$ and $J_1(j) \subseteq J_1(j+1)$. 

It follows that every element in 
\[M_{J}(\C)\sr = \{ a \in M_J(\C): a^* = \widetilde \gamma(a)\}\]
 is of the block form 
\[ \begin{pmatrix}
A & B &  \overline{B} \\
C & D & E \\ 
\overline C & \overline{E} & \overline{D}
\end{pmatrix},\]
where $A$ is real valued and $B,C, D, E$ are complex valued matrices.

We claim that $M_J(\C)\sr \cong M_J(\R)$. Set $h = |J|$, $h_1 = |J_f|$, and $h_2 =|J_1| =  |J_2|$ (so $h_1 + 2h_2 = h$). 
We know that (up to isomorphism) the only real $C\sp*$-algebras whose complexifications are isomorphic to $M_J(\C) = M_h(\C)$ are $M_h(\R)$ and $M_{h/2}(\Bbb H)$; and the second possibility can only happen if $h$ is even 
(see, for example, page 1 of \cite{schroderbook}). We show there exists a system of $h$ orthogonal projections in $M_J(\C)\sr$, which precludes the existence of an isomorphism 
$M_J(\C)\sr \cong M_{h/2}(\Bbb H)$.
Indeed, there are $h_1$ obvious orthogonal subprojections of 
\[ p =  \begin{pmatrix}
I_{h_1} & 0 &  0 \\
0 & 0 & 0 \\ 
0 & 0 & 0
\end{pmatrix} \; , \] \; 
and similarly there are $h_2$ orthogonal subprojections of each of
\[ q_1 = \begin{pmatrix}
0 & 0 &  0 \\
0 & 1/2 \,I_{h_2} & i/2 \, I_{h_2}  \\ 
0 & -i/2 \, I_{h_2} & 1/2 \,  I_{h_2}
\end{pmatrix} \quad \text{and} \quad
q_2 = \begin{pmatrix}
0 & 0 &  0 \\
0 & 1/2 \,I_{h_2} & -i/2 \, I_{h_2}  \\ 
0 & i/2 \, I_{h_2} & 1/2 \,  I_{h_2}
\end{pmatrix}   \; .   \] 
Notice that $p + q_1 + q_2 = I_h$. It follows that $M_J(\C)\sr $ is isomorphic to $M_h(\R)$.
Moreover, 
it is evident that this choice of orthogonal subprojections is compatible with the inclusion maps of the inductive limit $B_m(v) \cong \varinjlim M_J(\C)$.  Hence, if $ \gamma(v) = v$, $B_m(v)\sr \cong \mathcal K\sr(\ell^2(s^{-1}(v)))$ as claimed.

Now, suppose $ \gamma(v) = w\not= v$.  For any fixed $j \in \N$, $ \gamma$ is a bijection from 
$J_v:= \{\lambda \in \Lambda v : d(\lambda) \leq (j, j, \ldots, j)\}$ to $J_w := \{ \mu \in \Lambda w: d(\mu) \leq (j, \ldots, j)\}$.   
Therefore, for $(a,b) \in M_{J_v}(\C) \oplus M_{J_w}(\C) \subseteq B_m(v) \oplus B_m(w)$, the involution $\widetilde \gamma$ satisfies 
\[ \widetilde \gamma (a,b) = (b^t, a^t) \; ,   \]
and so the associated real matrix algebra is $\{ M \oplus \overline M: M \in M_{J_v}(\C) \} \cong M_{J_v}(\C)$.  As $ \mathcal K_\C(\ell^2(s^{-1}(v))) = \varinjlim_{j\to \infty} M_{J_v}(\C) \cong B_m(v) \oplus B_m(w)$ the result follows. 
\end{proof}

As a complement to the abstract reasoning above, and inspired by \cite[Theorem 2.5]{boersema-MJM}, we now exhibit a choice of basis for $\C^J$ which will give a more concrete argument for why $M_J(\C)\sr \cong M_J(\R)$ when $ \gamma(v) = v$.  Fix an arbitrary $n\in \Z^k$. 
For $\lambda \in J_f$ we define $t_{\lambda} := s_{\lambda, n - d(\lambda)}$, and if $\alpha \in J_1$ set 
\[ t_\alpha := \frac{s_{\alpha, n - d(\alpha)} + s_{ \gamma(\alpha), n - d(\alpha)}}{\sqrt 2}.\]
If $\beta \in J_2$ we define $t_\beta := \tfrac{i}{\sqrt 2}( s_{\gamma(\beta), n - d(\beta)} - s_{\beta, n - d(\beta)})$. One easily computes that 
\[ \widetilde \gamma(t_\lambda) = t_{\lambda}^*\]
for any $\lambda \in J$,
and that for any $ \alpha, \beta \in J$ we have 
\[ t_\beta^* t_\alpha = t_{\alpha}^* t_\beta  = \delta_{\alpha, \beta} s_{v, n}.\]
It follows that, for any $\alpha, \beta, \lambda, \eta \in J$, 
\[ t_\alpha t_\beta^* t_\lambda t_\eta^* = \delta_{\beta, \lambda} t_\alpha t_\eta^*.\]
In other words, $\{ t_\alpha t_\beta^*: \alpha, \beta \in J\}$ is a set of matrix units, which spans $M_J(\C)$ since $\{ s_{\lambda, n-d(\lambda)} {s_{\mu, n-d(\mu)}^*}: \lambda, \mu \in J\}$ does, and which satisfies $\widetilde \gamma(t_\alpha t_\beta^*) = t_{\beta} t_{\alpha}^*$ for all $\alpha, \beta \in J$.  With this basis, it is evident that $M_J(\C)\sr = M_J(\R).$

\begin{rmk} If $\Lambda$ is a directed graph (1-graph), the operators $\{ t_\alpha: \alpha \text{ an edge}\}$ were used in \cite[Theorem 2.4]{boersema-MJM} to show that any vertex-fixing involution $ \gamma$ on $\Lambda$ gives rise to the same real $C\sp*$-algebra as the trivial involution.  However, this proof breaks down in the $k$-graph case for $k > 1$, because $\{ t_\alpha: \alpha \in \Lambda\}$ need not satisfy the Cuntz--Krieger relations, even if all vertices are fixed by $ \gamma$.  In particular, if $ef \sim f'e'$ we need not have $t_e t_f = t_{f'} t_{e'}$.
It remains an open question whether the conclusion of \cite[Theorem 2.4]{boersema-MJM} extends to higher-rank graphs with involution.
\end{rmk} 

The following Corollary is immediate from Lemma \ref{lem:AF}.

\begin{cor} For each $m \in \Z^k$,
\[ B\pr_m \cong \bigoplus_{v \in G_f} \mathcal K_\R(\ell^2(s^{-1}(v))) \oplus \bigoplus_{v \in G_1} \mathcal K_\C(\ell^2(s^{-1}(v))) \;  \]
where $G_f$ is the set of vertices of $\Lambda$ that are fixed by $\gamma$ and $G_1$ is a set that contains exactly one vertex from every $\gamma$-orbit of cardinality $2$.
Consequently,
	 $ B\sr = C\sp*\sr(\Lambda \times_d \Z^k; \gamma) = \varinjlim B_m\pr$ is an AF real $C\sp*$-algebra.
	 \label{cor:AF}
\end{cor}

\subsection{The spectral sequence via group homology}
The  main result of this section is the following. 
\begin{thm} \label{sp-seq1}
There exists a spectral sequence $ \{ E^r, d^r \}$ of $\CR$-modules that converges to $K\crr( C\sp*\sr(\Lambda, \gamma) )$ and has
\[ E^2_{p,q} = H_p(\Z^k, K\crr(B\sr) ) \; .\]
\end{thm}

In this spectral sequence, each object $E_{p,q}^r$ is a $\CR$-module and each map $d^r_{p,q}$ is a $\CR$-module homomorphism. The spectral sequence is defined for all $p,q \in \Z$, but it is periodic in $q$. (The real part has period $8$ and the complex part has period $2$.) Also $E_{p,q}^r = 0$ for $p \notin \{0, 1, \dots, k\}$.

\begin{proof}[Proof of Theorem~\ref{sp-seq1}]
Let $k_*(B\sr)$ denote one of the graded functors $KO_*(B\sr)$ or $KU_*(B\sr)$. Applying \cite[6.10 Theorem]{kasparov-equivKK}
to the setting where $\pi = \Z^k$ and $D = B\sr$, we obtain a spectral sequence converging to the ``$\gamma$-part'' of $k_*(B\sr \rtimes \Z^k)$, 
and whose $E^1$ and $E^2$ pages are given by
\begin{align*}
E^1_{p,q} &\cong k_{p+q}(D_p/D_{p-1}) \cong \bigoplus_{m: 1 \leq m\leq  {k\choose p}}  
k_q(B\sr) \\
E^2_{p,q} & \cong H_p(\Z^k, k_q( B\sr ) )
\end{align*}
(Here 
 $0 \subseteq D_0 \subseteq D_1 \subseteq \cdots D_k = D_X$ is a filtration by ideals of a certain fixed-point algebra $D_X$, which is Morita equivalent to $B\sr$.) 
Since the Baum-Connes Conjecture with arbitrary coefficients is true for $\Z^k$ \cite{schick}, this
spectral sequence in fact converges precisely to $k_*(B\sr \rtimes \Z^k)$, which equals
$k_*(C\sr ^* (\Lambda, \gamma) )$ by Theorem \ref{B_times_Z}.
Taking both of these spectral sequences together, we have a spectral sequence with both a real and a complex part, that converges to $K\crr(C\sr ^* (\Lambda, \gamma) )$.

Now, let $k_*(B\sr), \widetilde k_*(B\sr)$ each independently denote one of the groups $KO_*(B\sr)$ or $KU_*(B\sr)$ and let
$\theta \colon k_*(B\sr) \rightarrow \widetilde{k}_*(B\sr)$ be one of the natural transformations $r, c, \eta, \beta, \psi $ of $K\crr(B\sr) = K\crr(D_X)$. 
Any one of these natural transformations can actually be represented by an element in $KK_*(C_1, C_2)$ where each $C_i$ is isomorphic to $\R$ or  $\C$. Multiplying by this $KK$-element induces the map 
$\theta^1: k_{p+q}(D_p/D_{p-1})  \to \widetilde k_{p+q}(D_p/D_{p-1})$.
Thus the $E^1$ page of the spectral sequence also has a natural $\CR$-structure. 
Furthermore, as observed by Schochet, the spectral sequence construction of {\cite{schochet-I}} is natural not only with respect to filtered homomorphisms of filtered $C\sp *$-algebras but also with respect to natural transformations of exact functors (see the comments on \cite[page 207]{schochet-I}). As Kasparov's spectral sequence construction follows that of Schochet, it follows that $\theta^{1} \colon E^1_*\rightarrow E^1_*$ commutes with the differentials and converges to the map 
$\theta^\infty \colon k_*( C\sr ^* (\Lambda, \gamma )) \rightarrow \widetilde k_*(C\sr ^* (\Lambda, \gamma) )$ induced by the original $KK$-element on the $E^\infty$ page. Therefore, we can consider the spectral sequence as a spectral sequence in the category of $\CR$-modules. 

At the $E^2$ page, we also have another $\CR$-module structure, induced by multiplying $k_*(B\sr)$ by the $KK$-element representing the natural transformation $\theta$.
It remains to show that the isomorphism $E^2_{p,q} \cong H_p(\Z^k, k_q(B\sr) )$ is a $\CR$-module isomorphism.
Recall from (\cite[p. 199]{kasparov-equivKK}) that, under the isomorphism
$E^1_{p,q}  \cong \bigoplus_m k_q(B\sr)$,
the differential map $d^1$ corresponds to the boundary homomorphism of the simplicial chain complex, yielding the isomorphism
$E^2_{p,q} \cong  H_p(\Z^k, k_*(B\sr) ) $.
It then follows immediately that
under this isomorphism, the map $\theta^2$ on $E^2$ which is induced from $\theta^1: k_{p+q}(D_p/D_{p-1}) \to \widetilde k_{p+q}(D_p/D_{p-1})$ is identical to the map on $H_p(\Z^k, k_*(B\sr) )$ which arises from $\theta\sp{B\sr}: k_*(B\sr) \to \widetilde k_*(B\sr)$ and the naturality of group homology 
(see for example Section III.6 of \cite{brownbook}).
Therefore, the $\CR$-module structure of $H_p(\Z^k, k_q(B\sr) )$ is the same as that of  $E^2_{p,q}$.
\end{proof}

\subsection{A combinatorial description of $E^2_{p,q}$}
	  \label{sec:combinatorial}
In this section, we will use the structure of  $B\sr$ as an AF algebra (Corollary \ref{cor:AF}) to obtain (in Theorem \ref{sp-seq2}) a more explicit formula for the  $E^2$ page of our spectral sequence from Theorem \ref{sp-seq1}.  To be precise, we identify a chain complex $\mathcal A^{(0)}$ whose $p$th group $\mathcal A_p^{(0)}$ consists of $k \choose p$ copies of a certain $\CR$-module, and whose homology computes $E^2_{p,q}$.  In Section \ref{sec:lowrank} below, we provide an explicit description of the connecting maps of this chain complex in terms of the adjacency matrices of $\Lambda$, in the situation where $k \leq 3$.

Recall that $B\sr = \varinjlim (B_m\pr, \iota_{nm})$, where (for $m \in \Z^k$)
\[B_m\pr = \overline{\text{span}} \{ s_{\mu, m- d(\mu)} s_{\nu, m- d(\nu)}^* \mid s(\mu) = s(\nu) = v\text{~for some $v \in \Lambda^0$ } \}  \subseteq B\sr = C\sp*\sr(\Lambda \times_d \Z^k) ,\]
 and $\iota_{nm} \colon B_m\pr \hookrightarrow B_n\pr$ (for $m \leq n \in \Z^k$) are the  connecting maps \eqref{eq:iota} of the inductive system. Let
$\fj_{nm} := (\iota_{nm})_* \colon K\crr(B\pr_m) \rightarrow K\crr(B\pr_n)$ be the induced map on united $K$-theory. 
Partition $\Lambda^0$ into three disjoint sets, $\Lambda^0 = G_f \sqcup G_1 \sqcup G_2$, where $ \gamma|_{G_f} = \id$ and $ \gamma(G_1) = G_2$.  With this notation, Corollary \ref{cor:AF} implies that 
\[ B\pr_m \cong \bigoplus_{v \in G_f} \mathcal K_\R(\ell^2(s^{-1}(v))) \oplus \bigoplus_{v \in G_1} \mathcal K_\C(\ell^2(s^{-1}(v))) \; \]
and consequently
\[ A_m := K\crr(B_m\pr) =  K\crr(\R)^{G_f} \oplus  K\crr(\C)^{G_1} \; .\]
The continuity of $K$-theory implies that
\[  A_\infty := \varinjlim (A_m, \fj_{nm}  ) \cong K\crr(B\sr) \; .\]

As in \cite[Section~3]{evans}, we define 
\[N_p = \{ (\mu_1, \dots, \mu_p)  \mid \mu_i \in \N, 1 \leq \mu_1  < \dots < \mu_p \leq k \} \; . \]
(The authors recognize that $\mu$ is also a common notation for an element of a $k$-graph $\Lambda$. We have chosen to follow Evans' notation, using $\mu_i$ and $\mu^i$ in reference to elements of $N_p$, for ease of cross-referencing.  It should always be clear from context (and the presence of sub- and super-scripts)  whether $\lambda$ or $\mu$ refers to an element of $N_p$ or of $\Lambda$.)

Observe that $|N_p| = {{k} \choose {p}}$.
If $\mu = (\mu_1, \dots, \mu_p) \in N_p$ then for any $1 \leq i \leq p$, we write
\[ \mu^i = \begin{cases} (\mu_1, \dots, \mu_{i-1}, \mu_{i+1}, \dots, \mu_p) \in N_{p-1} & \text{if~} p > 1\\
		\star & \text{if~} p =1 \; . \end{cases} \]
Let $\mathcal{B}$ denote the chain complex of $\CR$-modules,
\[ \mathcal{B}: ~~ 0 \rightarrow A_\infty \rightarrow \dots \rightarrow \bigoplus_{N_p} A_\infty \rightarrow \dots \rightarrow A_\infty \rightarrow 0 \;. \]  
Writing $\mathcal B_p :=   \bigoplus_{N_p} A_\infty $, the differentials 
$ \partial_p\colon 
\mathcal B_p \to \mathcal B_{p-1}$
are defined by
\begin{equation} \partial_p 
		= \bigoplus_{\lambda \in N_{p-1} } \sum_{\mu \in N_p} \sum_{i = 1}^p (-1)^{i+1} \delta_{\lambda, \mu^i} (\id - (\beta_{\mu_i})^{-1}_*) \; 
		\label{eq:partial}
\end{equation}
where $\beta$ denotes the usual action of $\Z^k$ on $B\sr$ given on generators by $\beta(n)s_{\mu, m} = s_{\mu, m+n}$.  We have $\beta_j = \beta(e_j)$, and we write $(\beta_j)_*$ for the induced map on $A_\infty$. 

For an element $y \in \mathcal B_p$, we can write $y = \bigoplus_{\mu \in N_p} y_\mu$ where $y_\mu \in A_\infty$. We find it convenient to write such an element alternatively as
$y = \sum_{\mu \in N_p} y_\mu e_\mu$ where $e_\mu \in \{0,1\}^{N_p}$ satisfies $e_\mu(\lambda) = \delta_{\mu, \lambda}$.
Using this notation, we can write the differentials of the complex $\mathcal{B}$ as
\[\partial_p (y_\mu e_\mu) = \sum_{i=1}^p (-1)^{i+1} (\id - (\beta_{\mu_i})^{-1}_*) (y_\mu) e_{\mu^i}  \quad \text{for $\mu \in N_p$ and $y_\mu \in A_\infty$.} \]

\begin{lemma} \label{E2=B}
There is a graded isomorphism
\[ H_*(\Z^k, K\crr(B\sr)) \cong H_*(\mathcal{B}) \; .\]
\end{lemma}

\begin{proof}
This result is proven exactly as in the proof of Lemma~3.12 of \cite{evans}, making use of the Koszul resolution for $\Z$ over $\Z G$ where $G = \Z^k$ and then tensoring that resolution by $k_*(B\sr)$ where the functor $k_*(-)$ is any of the functors $KO_i(-)$ and $KU_i(-)$.
\end{proof}

We now work towards a more concrete description of $H_*(\mathcal B)$.  For each $m \in \N^k$, let $\mathcal A^{(m)}$ denote the chain complex of $\CR$-modules
 \[ \mathcal{A}^{(m)} : ~~ 0 \rightarrow A_m \rightarrow \dots \rightarrow \bigoplus_{\mu \in N_p} A_m \rightarrow \dots \rightarrow A_m \rightarrow 0 \; \]
where, we recall, $A_m =  K\crr(B_m\pr )$. The differentials $\partial_p^{(m)}$ for $\A^{(m)}$ are defined by 
\[ \partial_p^{(m)} (y_\mu e_\mu) = \sum_{i=1}^p (-1)^{i+1} (\id - \phi^m_{\mu_i}) (y_\mu) e_{\mu^i}  \;  \quad \text{for $\mu \in N_p$ and $y_\mu \in A_m$} \]
where $\phi^m_j \colon K\crr(B_m\pr) \rightarrow K\crr(B_m\pr)$ is the map induced on $K\crr(-)$ by the composition
\[ B_m\pr \xrightarrow{\iota_{m+e_j,m}} B_{m+e_j}\pr \xrightarrow{\beta(-e_j)} B_m\pr. \]

Recall that $\fj_{nm}: A_m \to A_n$ is the map induced on $K$-theory by the inclusion map $\iota_{nm}: B_m \to B_n$ of \eqref{eq:iota}.
For each $m \leq n$, we extend the map $\fj_{nm} \colon A_m \rightarrow A_n$ to a chain map
$ \fJ_{n m} \colon \mathcal{A}^{(m)} \rightarrow \mathcal{A}^{(n)} $
defined by
\[ \fJ^p_{nm} \left(  \sum_{\mu \in N_{p} } y_\mu e_\mu  \right) =
		\sum_{\mu \in N_{p} } \fj_{nm}  \left( y_\mu \right) e_\mu \; .  \]

\begin{lemma} \label{lemma:chain_limit}
$\fJ_{nm}$ is a chain map for all $m \leq n$. Futhermore, there is an isomorphism of chain complexes $\mathcal{B} \cong \varinjlim ( \A^{(m)}, \fJ_{nm})$.
\end{lemma}

\begin{proof}
The claim that $\fJ_{nm}$ is a chain map is, by definition, the claim that the diagram
\[ \xymatrix{
\A_p^{(m)}			
\ar[r]^{\partial_p^{(m)}}		
\ar[d]^{\fJ_{nm}}	
& \A_{p-1}^{(m)}			
\ar[d]^{\fJ_{nm}}			\\							
\A_p^{(n)}	
\ar[r]^{\partial_p^{(n)}}
& \A_{p-1}^{(n)}	 	} \]
commutes for all $p$. Focusing on each summand of $\mathcal{A}_p^{(m)}= \bigoplus_{\mu \in N_p} A_m$, this is evidently equivalent to the commuting of the diagram
\[  \xymatrix{
A_m			
\ar[r]^{ \phi^m_{\mu_i} }	
\ar[d]^{\fj_{nm}}	
& A_m			
\ar[d]^{\fj_{nm}}			\\							
A_n
\ar[r]^{\phi^n_{\mu_i} }
& A_n	}\]
for all $i$. On the level of $C\sp*$-algebras, this follows from the relation 
\[ \beta(-e_j) \circ \iota_{n+e_j,m} = \iota_{n,m} \circ  \beta(-e_j) \circ \iota_{m+e_j, m},\] 
which holds thanks to the fact that every $\gamma \in \Lambda^{n-m+e_j}$ can be factored as 
$\gamma = \gamma_1 \gamma_2$ for a unique $\gamma_1 \in \Lambda^{e_j}, \gamma_2 \in \Lambda^{n-m}$.

Now we prove the second statement. Using the isomorphism $\varinjlim K\crr(B_m\pr) = K\crr(B\sr)$, we easily obtain $\varinjlim \A_p^{(m)} = \mathcal B_p$ for all $p$. For each $m \in \Z^k$, let $\fJ_m \colon \A^{(m)} \rightarrow \mathcal{B}$ be the map into the limit; this can also be described by
\[ \fJ_m^p \left( \sum_{\mu \in N_p} y_\mu e_\mu \right)
	= \sum_{\mu \in N_p} \fj_m (y_\mu e_\mu) \]
where $\fj_m \colon K\crr(B_m\pr) \rightarrow K\crr(B\sr)$ is the map induced by the inclusion $B_m\pr \hookrightarrow B\sr$.
It remains to show that the diagram
\[ \xymatrix{
\A_p^{(m)}			
\ar[r]^{\partial_p^{(m)}}		
\ar[d]^{\fJ_{m}}	
& \A_{p-1}^{(m)}			
\ar[d]^{\fJ_{m}}			\\							
\mathcal{B}_p
\ar[r]^{\partial_p}
& \mathcal{B}_{p-1}	}\] 
commutes, for which it suffices to show that
\[ \xymatrix{
A_m			
\ar[r]^{ \phi^m_{\mu_i} }	
\ar[d]^{\fj_{m}}	
& A_m			
\ar[d]^{\fj_{m}}			\\							
A_\infty
\ar[r]^{ (\beta_{\mu_i})_*^{-1} }
&A_\infty	} \]
commutes for all $i$. This follows from the definition of $\phi^m_{\mu_i}$ and the fact that the diagram
\[  \xymatrix{
B_m \ar[d]^{\iota_m} \ar[rr]^{ \beta(-e_i) }
&& B_{m-1}  \ar[d]^{\iota_{m-1}} \\
B  \ar[rr]^{ \beta(-e_i) }
&& B } \]
commutes on the level of $C\sp*$-algebras.
\end{proof}

The following lemma is the last key technical result that we need. The proof of the corresponding statement in the complex case is the bulk of the proof of Theorem~3.14 of \cite{evans}. The proof there is quite technical. Our proof will be so too, and here we have the additional complication of working in the category of $\CR$-modules, rather than the category of abelian groups. We mitigate some of this technicality through the use of the $e_\mu$ notation introduced above, as well as making explicit use of the concept of a chain homotopy, which Evans did not do.

In addition to the chain map $\fJ_{n m} \colon  \mathcal{A}^{(m)}  \rightarrow \mathcal{A}^{(n)} $ we also have the chain map   
$\fB_{nm} \colon  \mathcal{A}^{(m)}  \rightarrow \mathcal{A}^{(n)} $ for $m \neq n$ which is defined by the action $\beta(n-m)_* \colon K\crr(B_m) \rightarrow K\crr(B_n)$ extended to 
$\mathcal{A}^{(m)} = \bigoplus_{N_p} A_m = \bigoplus_{N_p} K\crr(B_m)$. 
It is routine to show that $\fB_{nm}$ is a chain map.
In fact, since $\beta(n-m) \colon B_n \rightarrow B_m$ is an isomorphism, $\fB_{nm}$ is an {isomorphism of chain complexes}.

\begin{lemma} \label{H_*=id}
For all $m \leq n$, 
the chain maps $\fB_{nm}$ and $\fJ_{nm}$ are chain homotopic. Thus
the induced map $(\fJ_{n m})_* \colon H_*( \mathcal{A}^{(m)} ) \rightarrow H_*( \mathcal{A}^{(n)} )$ is an isomorphism.
\end{lemma}

\begin{proof}
It suffices to prove the claim for $\fB_{m+e_j,m}$ and $\fJ_{m+e_j,m}$ for arbitrary $m, j$. 
We fix $m, j$ for the remainder of this proof and write $\fJ = \fJ_{m + e_j, m}$ and $\fB = \fB_{m+e_j, m}$.
For $\mu \in N_p$, let $\kappa(\mu)$ denote the cardinality of $\{ i \in \{1, \dots, p\} \mid \mu_i < j \}$.
Now let $\sigma^p \colon \A^{(m)}_p \rightarrow \A_{p+1}^{(m+e_j)}$ be the map defined by
\[ \sigma^p \left(   y_\mu e_\mu  \right)
	= \begin{cases} (-1)^{\kappa(\mu)} (\beta_j)_* (y_\mu) e_{\mu \cup \{j\} } & \text{if~} j \notin \mu \\
		0 & \text{if~} j \in \mu \;  \end{cases}  \; .\] 
This definition of $\sigma$ is inspired by the choice of $z$ in the proof of Theorem~3.14 in \cite{evans}.

We will show that for all $y \in \A^{(m)}_p$ we have (suppressing the superscripts for $\partial_p$)
\[ \partial_{p+1} \sigma^p (y) + \sigma^{p-1} \partial_p (y) = (\fB^p - \fJ^p) y \; , \]
so that $\sigma$ provides the desired chain homotopy between $\fB$ and $\fJ$.

It suffices by linearity to assume that $y = y_\mu e_\mu$ for some $\mu \in N_p$ and $y_\mu \in A_m$. First we consider the case $j \in \mu$; so $\mu_{\kappa(\mu) + 1} = j$. Then, writing  $\phi_j$ for $\phi^m_j$ and $\fj_j$ for $\fj_{m + e_j, m}$,
\begin{align*}
\partial_{p+1} \sigma^p (y_\mu e_\mu) &+ \sigma^{p-1} \partial_{p} (y_\mu e_\mu)  \\
	&= 0 + \sigma^{p-1} \left( \sum_{i=1}^p (-1)^{i+1} (\id - \phi_{\mu_i} ) (y_\mu) e_{\mu^i} \right) \\
	&= (-1)^{\kappa(\mu) } \sigma^{p-1} \left( (\id - \phi_j)(y_\mu) e_{\mu^{\kappa(\mu)+1}  } \right) 
						&&  \text{since $j \in \mu^i$ unless $i = \kappa(\mu) + 1$} \\ 
	&= (-1)^{\kappa(\mu) } (-1)^{\kappa(\mu)  }  (\beta_j)_* (\id - \phi_j)  (y_\mu) e_{\mu }  
						&&  \text{since $\kappa(\mu^{\kappa(\mu) + 1})  = \kappa(\mu) $} \\
	&=  ((\beta_j)_* - \fj_j)(y_\mu) e_{\mu }&& \text{since } \phi_j = (\beta_j)^{-1}_*  ( \iota_j)_* = (\beta_j)_*^{-1}  (\fj_j) \\
	&=  \left( \fB^p - \fJ^p \right) (y_\mu) e_{\mu }  \; .  
 \end{align*}
Now, consider the case $j \notin \mu$. Then
\begin{align*}
\sigma^{p-1} \partial_p (y_\mu e_\mu) 
	&= \sigma^{p-1} \left( \sum_{i =1}^p (-1)^{i+1} (\id- \phi_{\mu_i}) (y_\mu) e_{\mu^i} \right) \\
	&= \sum_{i = 1}^p (-1)^{i+1} (-1)^{\kappa(\mu^i) }  (\beta_j)_* (\id- \phi_{\mu_i})  (y_\mu) e_{\mu^i \cup \{j\}  } \\
\text{and~~~}
   \partial_{p+1} \sigma^p (y_\mu e_\mu) 
    	&= \partial_{p+1} \left(  (-1)^{\kappa(\mu)} (\beta_j)_* y_\mu e_{\mu \cup \{j\} }  \right)  \\
    	&= \sum_{i=1}^{p+1} (-1)^{\kappa(\mu) } (-1)^{i+1} (\id -\phi_{ (\mu \cup \{j\})_i }) (\beta_j)_* (y_\mu) e_{ \left( \mu \cup \{ j \} \right)^i  \; .}
\end{align*}

In this last sum, any term with $i \leq \kappa(\mu)$ is equal to
\[  (-1)^{\kappa(\mu)} (-1)^{i+1}  (\id - \phi_{\mu_i})  (\beta_j)_* (y_\mu) e_{\mu^i \cup \{j\} } 
				= (-1)^{\kappa(\mu^i) +1} (-1)^{i+1} (\id - \phi_{\mu_i}) (\beta_j)_* (y_\mu) e_{\mu^i \cup \{j\} }  \]
while any term with $i \geq \kappa(\mu) + 2$ is equal to
\[ (-1)^{\kappa(\mu)} (-1)^{i+1} (\id - \phi_{\mu_{i-1}} )  (\beta_j)_* (y_\mu) e_{\mu^{i-1} \cup \{j\} } 
				= (-1)^{\kappa(\mu ^{i-1})} (-1)^{i-1} (\id - \phi_{\mu_{i-1}} )  (\beta_j)_* (y_\mu) e_{\mu^{i-1} \cup \{j\} } \;. \]
{As the maps $\phi_{\mu_i}$ and $\beta_j^*$ 	commute for all $i, j$,}			
in the sum $\partial_{p+1} \sigma^p (y_\mu e_\mu) + \sigma^{p-1} \partial_{p}   (y_\mu e_\mu)$, all these terms cancel out, and the only term that remains is the summand of $\partial_{p+1} \sigma^p(y_\mu e_\mu)$ corresponding to  
$i = \kappa(\mu) + 1$. Therefore,
\begin{align*}
	 \partial_{p+1} \sigma^p (y_\mu e_\mu) + \sigma^{p-1} \partial_{p}   (y_\mu e_\mu)
	&=  (-1)^{ \kappa(\mu)} (-1)^{ \kappa(\mu)+2}   (\id - \phi_j) (\beta_j)_* (y_\mu) e_\mu \\
	&= (\id - \phi_j)  (\beta_j)_* (y_\mu) e_\mu \\
	&= \left( \fB^p - \fJ^p \right) (y_\mu) e_{\mu }. \qedhere
\end{align*}
\end{proof}

\begin{lemma} \label{B=A1}
$H_*(\mathcal{B}) \cong H_*( \mathcal{A}^{(0)}) $.
\end{lemma}

\begin{proof}
From Lemma~\ref{lemma:chain_limit} and the continuity of the homology functor, we have
$H_*(\mathcal{B}) = \lim_{m \to \infty} (H_*(\mathcal{A}^{(m)}), (\fJ_{nm})_* )$. However, Lemma~\ref{H_*=id} shows that the connecting maps of the limit are all isomorphisms.
Therefore $H_*(\mathcal{B}) \cong H_*( \mathcal{A}^{(0)}) $.
\end{proof}

\begin{thm} \label{sp-seq2}
Let $(\Lambda,  \gamma)$ be a $k$-graph with involution that is row-finite and has no sources. 
Then there exists a spectral sequence $\{ E^r, d^r  \}$ converging to $K\crr(C\sp*\sr(\Lambda,  \gamma))$ such that
$E_{p,q}^2 \cong H_p( \mathcal{A}^{(0)})$ and $E_{p,q}^{k+1} \cong E_{p,q}^{\infty}$.
\end{thm}

\begin{proof}
 Theorem~\ref{sp-seq1} gives the existence of the spectral sequence $\{E^r, d^r\}$. Lemmas~\ref{E2=B} and \ref{B=A1} combine to provide the isomorphism $E^2_{p,q} = H_p(\mathcal A^{(0)})$.  The isomorphism $E^{k+1}_{p,q} \cong E^\infty_{p,q}$ results from the fact that $E^2_{p,q} = H_p(\Z^k, k_q(B\sr )) = 0$ if $p\geq k+1$, so all of the differential maps $d^r_{p,q}$ are zero for $r \geq k+1$.
\end{proof}

\subsection{Notes on Computations using the Spectral Sequence}
\label{sec:lowrank}

We say that a $k$-graph $\Lambda$ is {\it finite} if the number of vertices is finite and the number of edges of degree $e_i$ is finite for each $i$. In this subsection we articulate Theorem~\ref{sp-seq2} more precisely in the specific cases of  a finite $k$-graph $\Lambda$ for $k = 1,2,3$.  That is, we identify the boundary maps of the chain complex $\mathcal A^{(0)}$, in order to  describe $K\crr(C\sp*\sr(\Lambda; \widetilde \gamma))$ in terms of the purely combinatorial data coming from the $k$-graph and its involution.

Throughout, we assume that $\Lambda$ is finite with involution $\widetilde \gamma$. We partition $\Lambda^0$ into three disjoint sets, $\Lambda^0 = G_f \sqcup G_1 \sqcup G_2$, where $\widetilde \gamma|_{G_f} = \id$ and $\widetilde \gamma(G_1) = G_2$.  Let $A$ denote the $\CR$-module
\[A := K\crr(\R)^{G_f} \oplus  K\crr(\C)^{G_1} .\] Recall that $A_0\pu = \Z^{\Lambda^0}$.  {Thanks to Theorem \ref{sp-seq2}, the $E^2$ page of our spectral sequence for $C\sp*\sr(\Lambda; \gamma)$ is given by the homology of the chain complex $\A^{(0)}$, all of whose component $\CR$-modules are direct sums of $A$.  

We first establish a handy lemma that will facilitate our description of the boundary maps of the chain complex $\mathcal A^{(0)}$.}
\begin{lemma} \label{CRMagic}
Let $M, N$ be two $\CR$-modules, which are each isomorphic to a finite direct sum of $K\crr(\R)$ and $K\crr(\C)$. Then any $\CR$-module homomorphism
$\alpha \colon M \rightarrow N$ is determined by the complex part $\alpha_0\pu$.
\end{lemma}

\begin{proof}
It suffices to consider the cases that $M$ is isomorphic to either $K\crr(\R)$ or to $K\crr(\C)$. 
Recall that the $\CR$-module $K\crr(\R)$ is free with a generator in the real part in degree 0 and $K\crr(\C)$ is free with a generator in the complex part in degree 0 (see Section~4.7 of \cite{bousfield90}). Thus the result is immediate in the case $M = K\crr(\C)$. 

Now suppose that $M = K\crr(\R)$ with generator $b \in M\po_0$. We must show that $\alpha\po_0(b)$ is uniquely determined by $\alpha\pu$.
We have $c(\alpha_0\po(b)) = \alpha\pu_0(c(b))$ where $c$ is the complexification map from $M\po$ to $M\pu$, or from $N\po$ to $N\pu$.
The complexification map $c$ in degree 0 is injective for both $K\crr(\R)$ and for $K\crr(\C)$. Thus, the formula $c(\alpha_0\po(b)) = \alpha\pu_0(c(b))$ determines $\alpha_0\po(b)$.
\end{proof}

Recall from Equation~\eqref{eq:matrices} that $M_i$ is the adjacency matrix of $\Lambda$ for the  edges of degree $e_i$. 
\begin{defn}
\label{def:phi-i}
For $1\leq i \leq k$, let $\rho^i \colon A \rightarrow A$ be the unique $\CR$-module homomorphism such that $(\rho^i)_0\pu \colon \Z^{\Lambda^0} \rightarrow \Z^{\Lambda^0}$ is represented by the matrix $B_i = \text{id} - M^t_i$. 
\end{defn}

\begin{rmk}
\label{rmk:beta-phi}
 Lemma \ref{lemma:chain_limit} above combines with 
\cite[Lemma 3.10]{evans} to reveal that the $\CR$-module homomorphism $(\beta_i^{-1})_*$ used in the definition of $\partial_p$ (see Equation \eqref{eq:partial} above) agrees with  $\rho^i$.  
\end{rmk}

Lemma~\ref{CRMagic} tells us that $(\rho^i)\po_j$ is completely determined by $(\rho^i)\pu_0$. The computation of 
$(\rho^i) \po_j$ from $(\rho^i)\pu_0$ follows the same method as indicated in \cite[Theorem 4.4]{boersema-MJM}.
In particular, if the complex part $(\rho^i) \pu_0 \in \en_\Z (\Z^{\Lambda^0})= \en_\Z (\Z^{G_f} \oplus \Z^{G_1} \oplus \Z^{G_2})$ is given by the matrix   $B_i = I - M_i^t$, then the functoriality of  $\widetilde \gamma$ implies that $\widetilde \gamma$ implements a bijection between the edges of color $i$ with source in $G_1$ and range in $G_2$, and the edges of color $i$ with source in $G_2$ and range in $G_1$.  Similarly, the edges with both source and range in $G_1$ are in bijection with the edges with source and range in $G_2$.  In other words, 
\[ B_i = \begin{pmatrix}  B_{11} & B_{12} & B_{12} \\ B_{21} & B_{22} & B_{23} \\ B_{21} & B_{23} & B_{22} \end{pmatrix}  \; . \]

It now follows that
the real part $(\rho^i) \po_0 \in 
 \en_\Z (\Z^{G_f} \oplus \Z^{G_1})$ is given by the matrix 
 \[ \begin{pmatrix} B_{11} & 2 B_{12} \\ B_{21} & B_{22} + B_{23} \end{pmatrix}  \;. \]
The other formulas for $(\rho^i) \po_j$ can be deduced from this easily; they are also given in \cite[Theorem 4.4]{boersema-MJM}.
For the convenience of the reader, we reproduce the relevant table in Figure \ref{fig1}.

\newcommand\TT{\rule{0pt}{2.6ex}} 
\newcommand\BB{\rule[-1.2ex]{0pt}{0pt}} 

\begin{center}
\begin{figure}
\label{fig1}
\begin{tabular}{|c|c| c  |  c  |  }
 \hline \hline
complex part & 0 \TT \BB& 
	$ \begin{pmatrix} B_{11}&  B_{12} & B_{12} \\ B_{21} & B_{22} & B_{23} \\ B_{21} & B_{23} & B_{22} \end{pmatrix}$ 
				& $  \Z^{|G_f|} \oplus \Z^{|G_1|} \oplus \Z^{|G_2|}  \rightarrow \Z^{|G_f|}  \oplus \Z^{|G_1|}\oplus \Z^{|G_2|} $ \\\cline{2-4}
& 1 \TT \BB &0 & $0$ \\
\hline \hline
real part  & 0 \TT \BB& $\begin{pmatrix} B_{11} &2 B_{12} \\ B_{21} & B_{22} + B_{23} \end{pmatrix}$
				& $  \Z^{|G_f|} \oplus \Z^{| G_1 |} \rightarrow \Z^{|G_f |} \oplus \Z^{| G_1 |}$ \\ \cline{2-4}
& 1 \TT \BB&  $B_{11}$ 
				& $ \Z_2^{|G_f|} \rightarrow \Z_2^{|G_f |}$ \\ \cline{2-4}
& 2 \TT \BB&  $\begin{pmatrix} B_{11} &  B_{12} \\ 0 & B_{22} - B_{23} \end{pmatrix} $
				& $  \Z_2^{|G_f|} \oplus \Z^{|G_1|} \rightarrow \Z_2^{|G_f |}  \oplus \Z^{|G_1|}$ \\ \cline{2-4}
& 3 \TT \BB &0 & $0$ \\ \cline{2-4}
& 4 \TT \BB& $\begin{pmatrix} B_{11} & B_{12} \\ 2B_{21} & B_{22} + B_{23} \end{pmatrix}$
				& $  \Z^{|G_f|} \oplus \Z^{| G_1 |} \rightarrow \Z^{|G_f|} \oplus \Z^{| G_1 |}$ \\ \cline{2-4}
& 5 \TT \BB&0 & $0$ \\  \cline{2-4}
& 6 \TT \BB& $B_{22} - B_{23}$ 
				& $ \Z^{|G_1|} \rightarrow \Z^{|G_1|}$ \\ \cline{2-4}
& 7 \TT \BB& 0& $0$ \\ 
 \hline \hline
\end{tabular}
\caption{Table for real $K$-theory}
\end{figure}
\end{center}

Once the maps $\rho^i$ are understood, Theorem~\ref{sp-seq2}  can be applied to develop the spectral sequence to compute $K\crr( C\sp*\sr( \Lambda, \widetilde \gamma))$. The following 
theorems articulate exactly how this looks in the cases $k = 1,2,3$. We note that for the case $k =1 $ we recover Theorem~4.1 of \cite{boersema-MJM}.

\begin{thm}[cf. Theorem~4.1 of \cite{boersema-MJM}]
Let $(\Lambda,  \gamma)$ be a finite $1$-graph with involution.
Then there is a 2-column spectral sequence that converges to $K\crr( C\sr^*(\Lambda,  \gamma))$ with 
$E_{p,q}^2$ equal to the homology of the chain complex $\mathcal A^{(0)}$,
\[ 0 \rightarrow A \xrightarrow{\partial_1} A \rightarrow 0, \]
where $\partial_1 = \rho^1$. 
\end{thm}
\begin{proof}
As $k=1$, we have $|N_0 | = |N_1| = 1$.  Therefore, in this case,
 Equation \eqref{eq:partial} simplifies to 
\[ \partial_1 =   \text{id} - (\beta_1)^{-1}_*  \, . \]
By Remark \ref{rmk:beta-phi}, $(\beta_1^{-1})_*$ agrees with the map whose complex part is represented by the matrix $M^t_1$. That is, $\partial_1 = \rho^1$.
\end{proof}

\begin{thm} \label{k=2--chaincomplex}
Let $(\Lambda,  \gamma)$ be a finite $2$-graph with involution.
Then there is a 3-column spectral sequence that converges to $K\crr( C\sp*\sr(\Lambda,  \gamma))$ with 
$E_{p,q}^2$ equal to the homology of the chain complex $\mathcal A^{(0)}$,
\[ 0 \rightarrow A \xrightarrow{\partial_2} A^2 \xrightarrow{\partial_1} A \rightarrow 0, \]
where
\begin{align*}
\partial_1 &= \begin{pmatrix} \rho^1 & \rho^2 \end{pmatrix} \\
\partial_2 &= \begin{pmatrix} -\rho^2 \\ \rho^1 \end{pmatrix}  \; .\\
\end{align*} 
\end{thm}
\begin{proof}
When $k=2$, we have $|N_1| = 2$ and $|N_2| = |N_0| =1$.  Therefore, Equation \eqref{eq:partial} and Remark \ref{rmk:beta-phi} tell us that $\partial_1: \A^2 \to \A$ and $\partial_2: \A \to \A^2$ are given by 
\[ \partial_1 = \sum_{\mu \in \{ 1, 2\}} (\text{id} - (\beta_\mu)^{-1}_*) = \begin{pmatrix} \rho^1 & \rho^2 \end{pmatrix}; \qquad \partial_2 = (-1)(\text{id} -( \beta_2)^{-1}_*) \oplus (\text{id} - (\beta_1)^{-1}_*) = \begin{pmatrix} -\rho^2 \\ \rho^1 \end{pmatrix}. \qedhere\]
\end{proof}

\begin{thm}
Let $(\Lambda,  \gamma)$ be a finite $3$-graph with involution.
Then there is a 4-column spectral sequence that converges to $K\crr( C\sr^*(\Lambda,  \gamma))$
with $E^2_{p,q}$ equal to the homology of the chain complex $\mathcal A^{(0)}$, 
\[ 0 \rightarrow A \xrightarrow{\partial_3} A^3 \xrightarrow{\partial_2} A^3 \xrightarrow{\partial_1} A \rightarrow  0 ,\]
where
\begin{align*}
\partial_1 &= \begin{pmatrix} \rho^1 & \rho^2 & \rho^3 \end{pmatrix} \\
\partial_2 &= \begin{pmatrix}  -\rho^2 & -\rho^3 & 0  \\ \rho^1 & 0 & - \rho^3 \\
						0 & \rho^1 & \rho^2 \end{pmatrix} \;  \\
\partial_3 &= \begin{pmatrix} \rho^3 \\ -\rho^2 \\ \rho^1 \end{pmatrix} \; . \\
\end{align*} 
\end{thm}
\begin{proof}
We justify the formula for $\partial_3$ and leave the remaining cases to the reader.  As $k= 3$, we have $|N_3| =1$ and $|N_2| = 3$. Write $N_3 = \{ \{ 1,2,3\} \}= \{\mu \}$. Given $1\leq i \leq 3$, there is a unique  $\lambda \in N_2$ with $\lambda = \mu^i$.  Ordering $N_2 = \{ \{ 1,2\}, \{1,3\}, \{2,3\} \}$ lexicographically, Equation \eqref{eq:partial} becomes
\[ \partial_3 = (-1)^{3+1}(\text{id} - (\beta_3)^{-1}_*) \oplus (-1)^{2+1}(  \text{id} - (\beta_2)^{-1}_*) \oplus (-1)^{1+1} (\text{id} - (\beta_1)^{-1}_*) = \begin{pmatrix} \rho^3 \\ -\rho^2 \\ \rho^1 \end{pmatrix}. \qedhere\]
\end{proof}

\section{Examples}
\label{sec:examples}
In this final section, we give three families of examples of real $C\sp*$-algebras that arise from rank-$2$ graphs with involution. 
{These examples showcase how one can leverage the $\CR$-module structure of real $K$-theory to completely determine $K\crr(C\sp*\sr(\Lambda, \gamma))$ on the basis of a  small amount of initial data.} In all three examples, our strategy follows the same general outline.  We begin by identifying the chain complex of Theorem \ref{k=2--chaincomplex} and computing its homology, which gives us the $E^2$ page of the spectral sequence.  {As $k=2$ in all of our examples, we have $E^\infty_{pq} = E^3_{pq}$ for all $p, q$; thus, our next step is to identify the differential $d^2$, which determines the $E^3= E^\infty$ page.  However, knowing the $E^\infty$ page does not completely describe $K\crr(C\sp*\sr(\Lambda, \gamma))$; rather, it gives a filtration (of at most 3 levels in the $k=2$ case) of $K\crr(C\sp*\sr(\Lambda, \gamma))$.

In our chosen examples, the $\CR$-module structure (and in particular the concept of the {\em core} of a $\CR$-module, as introduced by Hewitt in \cite{hewitt}) enable us to  describe $K\crr(C\sp*\sr(\Lambda, \gamma))$, up to at most two possibilities, using only the data from  the $E^2$ page. As the core will be a key tool in all of our computations in this section, we pause to discuss it in more detail.

To describe the core, recall that we have an  involution $\psi$ on $KU_*(C\sp*\sr(\Lambda, \gamma)) = K_*(C \sp *(\Lambda))$ which comes from the real structure on $C\sp*(\Lambda)$.
Moreover, since $KO_*(C\sp*\sr(\Lambda, \gamma))$ is a graded module over $KO_*(\R)$, for each $i$ we have $\eta_{i-1}: KO_{i-1}(C\sp*\sr(\Lambda, \gamma)) \to KO_i(C\sp*\sr(\Lambda,\gamma))$ which comes from multiplication by the nontrivial element in   $KO_1(\R) = \Z_2$.  Thus, we can define
\begin{align}
MO_i &= \rm{image}~\eta_{i-1} \colon KO_{i-1}(C\sr^*(\Lambda,  \gamma)) \rightarrow KO_{i}(C\sr^*(\Lambda,  \gamma)) \label{eq:core-gps}\\
\text{and} ~MU_i &= \frac{ \ker (1 -\psi_i)}{ \rm{image} (1 + \psi_i)}. \notag
\end{align}
Note that, since the $MO_i$ groups arise from the map $\eta$ which satisfies $2\eta = 0$, every $MO_i$ group is also 2-torsion. A straightforward computation will show that the $MU_i$ groups are also always 2-torsion.

The maps $\eta, c, r$ of $K\crr_*(C\sp*(\Lambda, \gamma))$ then naturally induce maps $\eta', c', r'$ on the groups $MO_i$ and $MU_i$, and we obtain a long exact sequence 
\begin{equation}
\label{eq:core-exact-seq}
 \cdots \rightarrow MO_i \xrightarrow{\eta'} MO_{i+1} \xrightarrow{c'} MU_{i} \xrightarrow{r'} MO_{i-2} \rightarrow \cdots \; \end{equation}
 (see \cite[Section 5.1]{hewitt} for details).
The core of the $\CR$-module $K\crr( C\sr^*(\Lambda,  \gamma) )$ is defined to consist of $KU_*(C\sr^*(\Lambda,  \gamma))$, the map $\psi$, and the groups and maps of the long exact sequence \eqref{eq:core-exact-seq}. 

Thus $KU_*(C\sr^*(\Lambda,  \gamma))$ is retained but $KO_*(C\sr^*(\Lambda,  \gamma))$ itself is dropped when we pass to the core; so on the face of it, we lose information. However, it follows from Theorem~4.2.1 of \cite{hewitt} that for two real $C\sp*$-algebras, $K\crr(A_1) \cong K\crr(A_2)$ if and only if the cores of $K\crr(A_1)$ and $K\crr(A_2)$ are isomorphic. Indeed, in our examples below, we compute some of the groups and maps in $K\crr(C\sr^*(\Lambda, \gamma))$ by using the spectral sequence, and then compute the core of the $\CR$-module to complete the identification of $K\crr(C\sr^*(\Lambda, \gamma))$. This saves the work of having to compute all of the groups of $KO_*(C\sr^*(\Lambda, \gamma))$ directly. 

 Notably, the factorization rules of the $k$-graph $\Lambda$ are irrelevant to the computations of the $E^2$ page.  Thus, the examples in this section support the conjecture \cite[Conjecture 5.11]{barlak-omland-stammeier} that the $K$-theory of a $k$-graph $C\sp*$-algebra should be (largely) independent of the choice of factorization rules.

In cases where we have multiple possibilities for $K\crr(C\sp*\sr(\Lambda, \gamma))$, the ambiguity comes from the fact that we have multiple possibilities for the $d^2$ map. In more complicated examples, it is also possible that the filtration of $K\crr(C\sp*\sr(\Lambda, \gamma))$ given on the $E^\infty$ page might not arise from a unique collection of $K$-theory groups.  We anticipate that a careful analysis of the impact of the factorization rules on $K\crr(C\sp*\sr(\Lambda, \gamma))$ may clarify these questions.}

The first family of examples we consider,  in Section~\ref{SectionExample1}, are $2$-graphs with only one vertex but an arbitrary number of edges of each type. In the second family of examples (Section~\ref{SectionExample2}) we consider 2-graphs with exactly three vertices, where the adjacency matrix is the same for both types of edges. Finally, in Section~\ref{SectionExample3}, we  consider a  family of 2-graphs with exactly three vertices but which have two distinct  adjacency matrices. Our computations result in a variety of different $\CR$-modules, many of which (but not all) have appeared in the literature before now or are direct sums of $\CR$-modules that have appeared before. 

All of the examples that we present have $K$-theory that is not consistent with a $1$-graph algebra, since they all have torsion in $KU_1( C\sp*(\Lambda))$ (see Corollary~4.3 of \cite{boersema-MJM}). In fact, in all of our examples, the complex $K$-theory is consistent with that of a tensor product of complex Cuntz algebras, $\mathcal{O}_m \otimes \mathcal{O}_n$. Therefore, in the case that the resulting real $C\sp*$-algebra is purely infinite and simple, \cite[Corollary 10.5]{brs} implies that they are all real forms of $\mathcal{O}_m \otimes \mathcal{O}_n$.

\subsection{A 1-vertex $2$-graph} \label{SectionExample1}

Let $\Lambda$ be a rank-2 graph with one vertex. Since all of the edges of degree $(1,0)$ and $(0,1)$ are just loops based at the vertex $v$, an involution $ \gamma$ on $\Lambda$ is just an involutive permutation on each of the two sets of loops, with the constraint that the permutation must be consistent with the factorization rules of $\Lambda$. 

In the special case that the factorization rules for $\Lambda$ are trivial and the involution $\gamma$ on $\Lambda$ is trivial, $C\sp*\sr(\Lambda,  \gamma) = C\sp*\sr(\Lambda, \id)$ is a tensor product of real Cuntz algebras:
\[C\sr^*(\Lambda,  \id) = C\sr^*(\Lambda) \cong C\sr^*(\Lambda_1 \times \Lambda_2) 
	= C\sr^*(\Lambda_1) \otimes\sr C\sr^*(\Lambda_2) \cong  \mathcal{O}_{m}\pr \otimes\sr \mathcal{O}_{n} \pr \]
by \cite[Corollary 3.5(iv)]{kp}. The $K$-theory for such tensor products of real Cuntz algebras is known from \cite{boersema2002}.
We will here compute $K\crr(C\sp*\sr(\Lambda, \gamma))$ more generally and find that essentially the same $K$-theory appears as in the tensor products, regardless of the factorization rules and the involution $\gamma$.

We first describe the specific $\CR$-modules that will arise, which we denote as $R_g$ for $g$ odd ($g \geq 3$)
and $S_g, T_g$ for $g$ even ($g \geq 2$). The groups in these $\CR$-modules are given below; in these examples, the natural transformations $r, c, \eta, \omega, \psi$ which complete the data of the $\CR$-module are completely determined by the given groups, and the relations among the homomorphisms mandated by the $\CR$-structure \eqref{CR Theorems} \eqref{CRMagic} and the long exact sequence \eqref{eq:CR-LES} linking the real and complex parts of a $\CR$-module.  The precise formulas for these natural transformations are recorded in \cite[Section 5.2]{boersema2002}.

\[ \begin{array}{|c|c|c|c|c|c|c|c|c|c|}  
\hline  \hline 
(g ~\text{odd}) & \makebox[1cm][c]{0} & \makebox[1cm][c]{1} & 
\makebox[1cm][c]{2} & \makebox[1cm][c]{3} 
& \makebox[1cm][c]{4} & \makebox[1cm][c]{5} 
& \makebox[1cm][c]{6} & \makebox[1cm][c]{7} \\
\hline  \hline
(R_g)\po_i
& \Z_g & \Z_g & 0 & 0 & \Z_g & \Z_g & 0 & 0  \\
\hline  
(R_g)\pu_i
& \Z_{g} & \Z_{g} & \Z_{g} & \Z_{g} & \Z_{g} & \Z_{g} & \Z_{g} & \Z_{g}\\
\hline  
\hline
\end{array} \]

\[\begin{array}{|c|c|c|c|c|c|c|c|c|c|}  
\hline  \hline 
(g ~ \text{even}) & \makebox[1cm][c]{0} & \makebox[1cm][c]{1} & 
\makebox[1cm][c]{2} & \makebox[1cm][c]{3} 
& \makebox[1cm][c]{4} & \makebox[1cm][c]{5} 
& \makebox[1cm][c]{6} & \makebox[1cm][c]{7} \\
\hline  \hline
(S_g)\po_i
& \Z_g & \Z_{2g} & \Z_2^2 & \Z_2^2 & \Z_{2g} & \Z_g & 0 & 0  \\
\hline  
(S_g)\pu_i
& \Z_{g} & \Z_{g} & \Z_{g} & \Z_{g} & \Z_{g} & \Z_{g} & \Z_{g} & \Z_{g}\\
\hline  
\hline
\end{array} \]

\[ \begin{array}{|c|c|c|c|c|c|c|c|c|c|}  
\hline  \hline 
(g \equiv 0 ~\text{mod}~ 4) & \makebox[1cm][c]{0} & \makebox[1cm][c]{1} & 
\makebox[1cm][c]{2} & \makebox[1cm][c]{3} 
& \makebox[1cm][c]{4} & \makebox[1cm][c]{5} 
& \makebox[1cm][c]{6} & \makebox[1cm][c]{7} \\
\hline  \hline
(T_g)\po_i
& \Z_g & \Z_2 \oplus \Z_g & \Z_2^3 & \Z_2^3 & \Z_2 \oplus \Z_g & \Z_{g} & 0 &0  \\
\hline  
(T_g)\pu_i
& \Z_{g} & \Z_{g} & \Z_{g} & \Z_{g} & \Z_{g} & \Z_{g} & \Z_{g} & \Z_{g}\\
\hline  
\hline
\end{array} \]

\[ \begin{array}{|c|c|c|c|c|c|c|c|c|c|}  
\hline  \hline 
(g \equiv 2 ~\text{mod}~ 4) & \makebox[1cm][c]{0} & \makebox[1cm][c]{1} & 
\makebox[1cm][c]{2} & \makebox[1cm][c]{3} 
& \makebox[1cm][c]{4} & \makebox[1cm][c]{5} 
& \makebox[1cm][c]{6} & \makebox[1cm][c]{7} \\
\hline  \hline
(T_g)\po_i
& \Z_g & \Z_2 \oplus \Z_g & \Z_2 \oplus \Z_4 & \Z_2 \oplus \Z_4 & \Z_2 \oplus \Z_g & \Z_{g} & 0 &0  \\
\hline  
(T_g)\pu_i
& \Z_{g} & \Z_{g} & \Z_{g} & \Z_{g} & \Z_{g} & \Z_{g} & \Z_{g} & \Z_{g}\\
\hline  
\hline
\end{array} \]

For later reference during the calculations in this section, we also record the groups $MO_i$ and $MU_i$ corresponding to the $\CR$-modules $S_g$ and $T_g$ in the tables below. Recall that the core of a $\CR$-module $M$ consists of just the complex part of $M$ and the groups of $MO_i$ and $MU_i$ (and the relevant natural transformations).
For $R_g$, we will not make use of the core but we note for completeness that $MO_i = 0$ and $MU_i = 0$ for all $i$.

\[\begin{array}{|c|c|c|c|c|c|c|c|c|c|}  
\hline  \hline 
\text{Core of~} S_g \text{~for~} g ~\text{even} & \makebox[1cm][c]{0} & \makebox[1cm][c]{1} & 
\makebox[1cm][c]{2} & \makebox[1cm][c]{3} 
& \makebox[1cm][c]{4} & \makebox[1cm][c]{5} 
& \makebox[1cm][c]{6} & \makebox[1cm][c]{7} \\
\hline  \hline
MO_i
& 0 & \Z_2 & \Z_2^2 & \Z_2^2 & \Z_2^2 & \Z_2 & 0 &  0  \\ \hline
MU_i
& \Z_2 & \Z_2 & \Z_2 & \Z_2 & \Z_2 & \Z_2 & \Z_2 & \Z_2  \\
\hline  
\hline
\end{array} \]

\[ \begin{array}{|c|c|c|c|c|c|c|c|c|c|}  
\hline  \hline 
\text{Core of~} T_g \text{~for~} g ~\text{even}  & \makebox[1cm][c]{0} & \makebox[1cm][c]{1} & 
\makebox[1cm][c]{2} & \makebox[1cm][c]{3} 
& \makebox[1cm][c]{4} & \makebox[1cm][c]{5} 
& \makebox[1cm][c]{6} & \makebox[1cm][c]{7} \\
\hline  \hline
MO_i 
& 0 & \Z_2 & \Z_2 & \Z_2^2 & \Z_2 & \Z_2 & 0 &  0  \\ \hline
MU_i 
& \Z_2 & \Z_2 & \Z_2 & \Z_2 & \Z_2 & \Z_2 & \Z_2 & \Z_2  \\
\hline  
\hline
\end{array} \]

Given $m, n \in \N_{\geq 2}$, define  $g = \gcd(m-1, n-1)$. From Section~5.2 of \cite{boersema2002} we have 
\begin{equation}
\label{eq:K-tensor}
K\crr( \mathcal{O}_m\pr \otimes\sr \mathcal{O}_n\pr) \cong \begin{cases} R_g & \text{$g$ odd}  \\ 
			S_g & m-1 \equiv n-1 \equiv 2 \pmod 4 \\ 
			T_g & m-1 \equiv 0  \text{~or~} n-1 \equiv 0 \pmod 4. \\     \end{cases}
\end{equation}
In particular, there  are isomorphisms $R_g \cong K\crr(\mathcal{O}\pr_{g+1} \otimes\sr \mathcal{O} \pr_{g+1})$ for $g$ odd;
$T_g \cong K\crr(\mathcal{O}\pr_{g+1} \otimes\sr \mathcal{O} \pr_{g+1})$ for $g \equiv 0 \pmod 4$; and $S_g \cong K\crr(\mathcal{O}\pr_{g+1} \otimes\sr \mathcal{O} \pr_{g+1})$ for $g \equiv 2 \pmod 4$.

\begin{prop}
\label{prop:one-vertex-2-graph}
Let $\Lambda$ be a rank-2 graph with one vertex. Let $m$ be the number of edges of degree $(1,0)$ and let $n$ be the number of edges of degree $(0,1)$. Assume $m, n \geq 2$. Let ${\gamma}$ be an involution on $\Lambda$, and write $g = \gcd(m-1,n-1)$. Then
\[ K\crr( C\sr^*(\Lambda,  \gamma)) \cong \begin{cases}  
	R_g & \text{~if $g$ is odd}, \\
	S_g \text{~or~}  T_g & \text{~if $g$ is even.} \end{cases} \]
\end{prop}
Before we begin the proof of Proposition \ref{prop:one-vertex-2-graph}, we pause to make a few comments.   First, note that if $g$ is odd, then $K\crr(C\sr^*(\Lambda,  \gamma))$ depends only on the number of edges of each color, not on the choice of involution or the  factorization rules defining $\Lambda$.  In particular,
Proposition \ref{prop:one-vertex-2-graph}  gives more evidence in support of \cite[Conjecture 5.11]{barlak-omland-stammeier}, which asserts that the $K$-theory of a one-vertex $k$-graph $C\sp*$-algebra should be independent of the factorization rules.

We also wish to remark on the uncertainty of the  statement  of Proposition \ref{prop:one-vertex-2-graph}
regarding the even case.  As $\mathcal O_n$ is the graph $C\sp*$-algebra of the one-vertex graph $E_n$ with $n$   edges, \cite[Corollary 3.5]{kp} tells us that there exists a $2$-graph 
$\Lambda = E_n \times E_m$ such that $|\Lambda^{(1,0)}|=m, |\Lambda^{(0,1)}| = n$ and $C\sp*(\Lambda) \cong \mathcal O_n \otimes \mathcal O_m$.  Therefore, applying Equation \eqref{eq:K-tensor} to $C\sp*\sr(\Lambda,  \id)$,
we see that both $S_g$ and $T_g$ can appear as the $K$-theory of a 2-graph of the type discussed in Proposition \ref{prop:one-vertex-2-graph}. However, we will see in the calculation below that in general it is not clear how to determine which $\CR$-module appears from the spectral sequence. 

We also have the following Corollary to Proposition \ref{prop:one-vertex-2-graph}.
 \begin{cor}
Fix $m, n \in \N_{\geq  2}$; a one-vertex 2-graph $\Lambda$ with $ |\Lambda^{(1,0)}|=m, |\Lambda^{(0,1)}| = n $; and an involution $ \gamma $ on $\Lambda$. If $g = \gcd(m-1, n-1)$ is odd and $C\sp*(\Lambda)$ is simple, then $C\sp*\sr(\Lambda,  \gamma) \cong C\sp*\sr(\Lambda, \rm{id})$.
 \end{cor}
 \begin{proof}
Recall from  \cite[Proposition 4.8]{kp} (cf.~also \cite[Lemma 3.2]{robertson-sims}) that the factorization rules which define $\Lambda$ will determine  if $C\sp*(\Lambda)$ is simple. 
 When $C\sp*(\Lambda)$ is simple, \cite[Corollary 5.1]{brown-clark-sierakowski} tells us that since $m, n \geq 2$, 
 $C\sp*(\Lambda)$ is purely infinite.  Consequently by  \cite[Theorem 10.2]{brs}, $C\sp*\sr(\Lambda, \gamma)$ is classified by its $K$-theory.  
 
 If $g$ is odd, then  Proposition \ref{prop:one-vertex-2-graph} tells us that this $K$-theory is independent of the involution $ \gamma$, so 
 \[C\sp*\sr(\Lambda,  \gamma) \cong C\sp*\sr(\Lambda, \gamma_{triv})\cong \mathcal O_m^\R \otimes \mathcal O_n^\R\]
 for any involution $ \gamma$ on $\Lambda$.
 \end{proof}

We now undertake the proof of Proposition \ref{prop:one-vertex-2-graph}.

\begin{proof}[Proof of Proposition \ref{prop:one-vertex-2-graph}]
The incidence matrices are $1 \times 1$ matrices, so $1 - M_1^t = 1-n$ and $1 - M_2^t = 1-m$. As $\Lambda^0 = \{ v\} = G_f$, we have $K\crr(B\sr) = A = K\crr(\R)$. Theorem \ref{k=2--chaincomplex} therefore tells us that the chain complex $\mathcal A^{(0)}$ is
\begin{equation} 
0 \rightarrow K\crr(\R) \xrightarrow{
\left( \begin{smallmatrix} -\rho^2 \\ \rho^1 \end{smallmatrix}  \right)} 
	K\crr(\R)^{2} \xrightarrow{ \left( \begin{smallmatrix} \rho^1 &   \rho^2 \end{smallmatrix} \right) }
	K\crr(\R) \rightarrow 0 \; . 
	\label{eq:one-vx-chain-cx}
	\end{equation}
We now use Table \ref{fig1} to compute the individual maps $\rho^i_j$ in each degree $j$.
\begin{align*}
~& \text{ \underline{complex part:}} \\
\text{degree $0$:} \qquad &0 \rightarrow \Z  \xrightarrow{ \left( \begin{smallmatrix} m-1 \\ 1-n \end{smallmatrix}  \right) }
			 ~\Z^2 \xrightarrow{ \left( \begin{smallmatrix} 1-n &   1-m \end{smallmatrix} \right)  } \Z \rightarrow 0 \\
~ \\	 
~& \text{ \underline{ real part:}} \\
\text{degree $0$:} \qquad &0 \rightarrow \Z  \xrightarrow{ \left( \begin{smallmatrix} m-1 \\ 1-n \end{smallmatrix}  \right) }
			 ~\Z^2 \xrightarrow{ \left( \begin{smallmatrix} 1-n &   1-m \end{smallmatrix} \right)  } \Z \rightarrow 0 \\
\text{degree $1$:} \qquad &0 \rightarrow \Z_2  \xrightarrow{ \left( \begin{smallmatrix} m-1 \\ 1-n \end{smallmatrix}  \right) }
			 ~ \Z_2^2 \xrightarrow{ \left( \begin{smallmatrix} 1-n &   1-m \end{smallmatrix} \right)  } \Z_2 \rightarrow 0 \\
\text{degree $2$:} \qquad &0 \rightarrow \Z_2  \xrightarrow{ \left( \begin{smallmatrix} m-1 \\ 1-n \end{smallmatrix}  \right) }
			 ~\Z_2^2 \xrightarrow{ \left( \begin{smallmatrix} 1-n &   1-m \end{smallmatrix} \right)  } \Z_2 \rightarrow 0 \\
\text{degree $4$:} \qquad &0 \rightarrow \Z  \xrightarrow{ \left( \begin{smallmatrix} m-1 \\ 1-n \end{smallmatrix}  \right) }
			 ~ \Z^2 \xrightarrow{ \left( \begin{smallmatrix} 1-n  &  1-m \end{smallmatrix} \right)  } \Z \rightarrow 0 \\	
\end{align*}
(For any degree not shown, the sequence consists of all trivial groups.)

The $E^2$ page of the spectral sequence has both a real part and a complex part, denoted $(E^2_{i,j})\po$ and $(E^2_{i,j})\pu$, the groups of which are derived from the chain complex above.
In the case that $g  = \gcd (m-1, n-1)$ is odd, we use the first line of the table above to compute that
\begin{equation*}
 (E^2_{i,j})\pu = H_i( (\mathcal{A}^{(0)})_j \pu) 
   = \begin{cases}
  \Z_g, &  \text{ if $i = 0,1$ and $j$ even }\\
  0, & \text{ otherwise.}
  \end{cases}
  \end{equation*}

The other lines of the table reveal that 
\begin{equation}
\label{eq:hlogy-1-vx} (E^2_{i,j})\po = H_i( (\mathcal{A}^{(0)})_j \po) =  \begin{cases}  \Z_{g}  & \text{ if $i = 0,1$ and $j \equiv 0 \pmod 4$ }\\
				0 & \text{otherwise.} \end{cases} 
				\end{equation}

From this data, we obtain the $E^2$ page of the  spectral sequence which converges to $K\crr(C\sp*\sr(\Lambda, \gamma))$.  The left-hand diagram below is the $E^2$ page for the complex $K$-theory and the right-hand diagram is for the real $K$-theory. Notice that the $j$ index is vertical and the $i$ index is horizontal.
The spectral sequence is 0 in all non-pictured columns, and is periodic with period 8 in the vertical direction.

\begin{gather*}
 \text{ \underline{ $E^2_{p,q}$ when $g$ is odd}} \\
\def\vvline{\hfil\kern\arraycolsep\vline\kern-\arraycolsep\hfilneg}
\begin{array}{ cccc }
 \multicolumn
{4}{c}{ \underline{\text{complex part}}} \\
\vspace{.25cm} \\
~ \vvline& \hspace{.3cm} \vdots \hspace{.3cm} &\hspace{.3cm} \vdots \hspace{.3cm} & \hspace{.3cm} \vdots \hspace{.3cm} \\
7  \vvline & 0 & 0 & 0  \\
6   \vvline  &  \Z_g &  \Z_g  & 0 \\
5   \vvline  &  0    &  0   &  0    \\
4   \vvline  &  \Z_{g} & \Z_{g}  & 0   \\
3    \vvline &  0  & 0 &  0   \\
 2   \vvline  &  \Z_g  &  \Z_g   & 0  \\
 1   \vvline  &  0 &  0  &0  \\
0 \vvline & \Z_g & \Z_g & 0 \\ \hline
~ \vvline & 0 & 1 & 2
\end{array}
\hspace{2cm}
\begin{array}{ cccc }
 \multicolumn
{4}{c}{ \underline{\text{real part}}} \\
\vspace{.25cm} \\
~ \vvline & \hspace{.3cm} \vdots \hspace{.3cm} &\hspace{.3cm} \vdots \hspace{.3cm} & \hspace{.3cm} \vdots \hspace{.3cm} \\
7  \vvline & 0 & 0 & 0  \\
6   \vvline  &  0    &  0   &  0    \\
5   \vvline  &  0    &  0   &  0    \\
4   \vvline  &  \Z_g &  \Z_g  & 0 \\
3    \vvline &  0  & 0 &  0   \\
 2   \vvline  & 0    &  0   &  0    \\
 1   \vvline  &  0 &  0  &0  \\
0 \vvline & \Z_g & \Z_g & 0 \\ \hline
~ \vvline & 0 & 1 & 2
\end{array}
\end{gather*}

The $d^2$ map has degree $(-2, 1)$ and is hence equal to 0 everywhere. It follows that $E^2 = E^\infty$ and that
$KO_q( C\sr^*(\Lambda, \gamma) ) $ has a filtration whose factors are the groups in the rightmost table above whose $i$ and $j$ coordinates sum to $q$.  Since, for each $q$, there is at most one nonzero such group, we conclude that
\[ KO_q(C\sr^*(\Lambda, \gamma)) =  \begin{cases}
\Z_{g} & q \equiv 0, 1, 4, 5 \pmod 8 \\
0 & q\equiv 2,3,6,7 \pmod 8. \end{cases} \]
Similarly, $KU_q(C\sr^*(\Lambda, \gamma)) = \Z_g$ for all $q$.
Therefore $K\crr(C\sr^*(\Lambda;\gamma)) \cong R_g$ if $g$ is odd.

Now consider the case that $g$ is even (with $m, n \geq 3$). The computations for $KU_*(C\sp*\sr(\Lambda, \gamma))$ are the same as in the odd case above.  When computing {$H_i(( \A^{(0)})^O_j)$} 
 for even $g$,  we obtain nearly the same formulas as we found in Equation \eqref{eq:hlogy-1-vx} for  the case that $g$ is odd.  The difference arises from the fact that all of the maps in the real part of the chain complex \eqref{eq:one-vx-chain-cx}
in degrees 1 and 2 are zero if $g$ is even. Hence, 
the $E^2$ page of the spectral sequence which converges to $KO_*(C\sp*\sr(\Lambda, \gamma))$ when $g$ is even is as shown:

\begin{gather*}
 \text{ \underline{ $E^2_{p,q}$ when $g$ is even}} \\
\def\vvline{\hfil\kern\arraycolsep\vline\kern-\arraycolsep\hfilneg}
\begin{array}{ cccc }
 \multicolumn
{4}{c}{ \underline{\text{complex part}}} \\
\vspace{.25cm} \\
~ \vvline& \hspace{.3cm} \vdots \hspace{.3cm} &\hspace{.3cm} \vdots \hspace{.3cm} & \hspace{.3cm} \vdots \hspace{.3cm} \\
7  \vvline & 0 & 0 & 0  \\
6   \vvline  &  \Z_g &  \Z_g  & 0 \\
5   \vvline  &  0    &  0   &  0    \\
4   \vvline  &  \Z_{g} & \Z_{g}  & 0   \\
3    \vvline &  0  & 0 &  0   \\
 2   \vvline  &  \Z_g  &  \Z_g   & 0  \\
 1   \vvline  &  0 &  0  &0  \\
0 \vvline & \Z_g & \Z_g & 0 \\ \hline
~ \vvline & 0 & 1 & 2
\end{array}
\hspace{2cm}
\begin{array}{ cccc }
 \multicolumn
{4}{c}{ \underline{\text{real part}}} \\
\vspace{.25cm} \\
~ \vvline & \hspace{.3cm} \vdots \hspace{.3cm} &\hspace{.3cm} \vdots \hspace{.3cm} & \hspace{.3cm} \vdots \hspace{.3cm} \\
7  \vvline & 0 & 0 & 0  \\
6   \vvline  &  0    &  0   &  0    \\
5   \vvline  &  0    &  0   &  0    \\
4   \vvline  &  \Z_g &  \Z_g  & 0 \\
3    \vvline &  0  & 0 &  0   \\
 2   \vvline  & \Z_2    &  \Z_2^2   &  \Z_2    \\
 1   \vvline  & \Z_2    &  \Z_2^2   &  \Z_2    \\
0 \vvline & \Z_g & \Z_g & 0 \\ \hline
~ \vvline & 0 & 1 & 2
\end{array}
\end{gather*}

The map $d^2$ again is equal to 0 everywhere except possibly the map $(d^2_{(2,1)})^O \colon \Z_2 \rightarrow \Z_2$ from degree $(2,1)$ to degree $(0,2)$ -- this map may or may not be the zero map. Then the $E^3$ page of the spectral sequence must look like one of the following. The left version corresponds to the case $(d^2_{(2,1)})^O = 0$ and the right version corresponds to the case $(d^2_{(2,1)})^O \neq 0$. For $n \geq 3$ we have $d^3 = 0$, thus $E^\infty = E^3$. 

\begin{gather*}
\text{ \underline{ real part of $E^\infty_{p,q} $ }}\\
\def\vvline{\hfil\kern\arraycolsep\vline\kern-\arraycolsep\hfilneg}
\begin{array}{ cccc }
 \multicolumn
{4}{c}{ \underline{ \text{case 1: $d^2_{2,1} = 0$}}} \\
\vspace{.25cm} \\
~ \vvline& \hspace{.3cm} \vdots \hspace{.3cm} &\hspace{.3cm} \vdots \hspace{.3cm} & \hspace{.3cm} \vdots \hspace{.3cm} \\
7  \vvline & 0 & 0 & 0  \\
6   \vvline  &  0 &  0  & 0 \\
5   \vvline  &  0    &  0   &  0    \\
4   \vvline  &  \Z_{g} & \Z_{g}  & 0   \\
3    \vvline &  0  & 0 &  0   \\
 2   \vvline  &  \Z_2  &  \Z_2^2   & \Z_2  \\
 1   \vvline  &   \Z_2  &  \Z_2^2   & \Z_2  \\
0 \vvline & \Z_g & \Z_g & 0 \\ \hline
~ \vvline & 0 & 1 & 2
\end{array}
\hspace{2cm}
\begin{array}{ cccc }
 \multicolumn
{4}{c}{ \underline{ \text{case 2: $d^2_{2,1} \neq 0$}}} \\
\vspace{.25cm} \\
~ \vvline & \hspace{.3cm} \vdots \hspace{.3cm} &\hspace{.3cm} \vdots \hspace{.3cm} & \hspace{.3cm} \vdots \hspace{.3cm} \\
7  \vvline & 0 & 0 & 0  \\
6   \vvline  &  0    &  0   &  0    \\
5   \vvline  &  0    &  0   &  0    \\
4   \vvline  &  \Z_g &  \Z_g  & 0 \\
3    \vvline &  0  & 0 &  0   \\
 2   \vvline  & 0   &  \Z_2^2   &  \Z_2    \\
 1   \vvline  & \Z_2    &  \Z_2^2   &  0    \\
0 \vvline & \Z_g & \Z_g & 0 \\ \hline
~ \vvline & 0 & 1 & 2
\end{array}
\end{gather*}

Once the $E^\infty$ groups are settled, this determines $KO_q(C\sp*\sr(\Lambda, \gamma))$ only ``up to extensions'', meaning that there is a filtration of $KO_q(C\sp*\sr(\Lambda, \gamma))$ in which the successive subquotients are isomorphic to 
$E^\infty_{i,j}$ where $i + j = q$. However,
we can deduce some specific information from the spectral sequence, namely that  $KO_n(C\sr^*(\Lambda,  \gamma)) = 0$ for $n = 6,7$, and that $KO_5(C\sp*\sr(\Lambda, \gamma)) = KO_0(C\sp*\sr(\Lambda, \gamma)) = \Z_g$.

To complete the computation of $K\crr(C\sp*\sr(\Lambda, \gamma))$, we now consider the core. 
We claim that the involution $\psi_*$ induced on $KU_*( C\sp*\sr(\Lambda, \gamma))$ by the real structure of $C\sp*\sr(\Lambda, \gamma)$ satisfies $\psi_j = 1$ for $j = 0,1,4,5$ and $\psi_j = -1$ for $j = 2,3,6,7$. To prove this claim, we first observe that for the $\CR$-module $K\crr(\R)$, we have $\psi_\R = 1$ in degree 0.  In the complex part 
\begin{equation}
0 \rightarrow \Z \xrightarrow{ \left( \begin{smallmatrix} m-1 \\ 1-n \end{smallmatrix}  \right) }
			 ~\Z^2 \xrightarrow{ \left( \begin{smallmatrix} 1-n &   1-m \end{smallmatrix} \right)  }\Z \rightarrow 0
\label{eq:one-vx-complex-hlogy}			 
			 \end{equation}
 of the chain complex $\mathcal A^{(0)}$, each copy of $\Z$ represents $KU_0(\R)$.  Thus,  $\psi_\R$ induces the  identity map on the homology groups $H_0, H_1$ of the chain complex \eqref{eq:one-vx-complex-hlogy}.  As $H_0 = (E^2_{0,0})^U = KU_0(C\sp*\sr(\Lambda, \gamma))$ and $H_1 = (E^2_{1,0})^U = KU_1(C\sp*\sr(\Lambda, \gamma))$, we conclude that $\psi_j = 1$ for $j = 0,1$.  The relation $\psi_{j+2} \beta = - \beta \psi_j$ then implies that, as claimed,
 \[ \psi_j = 1 \text{ for }j = 0,1,4,5\quad \text{  and } \quad \psi_j = -1\text{ for }j = 2,3,6,7.\] 

Recall that  $MU_i = ( \ker (1 -\psi_i) )/( \rm{image} (1 + \psi_i))$.
Since $g$ is even, one computes that $MU_0 = \Z_g/2\Z_g \cong \Z_2$, and $MU_2 = \{ 0, g/2\}/\{ 0\} \cong \Z_2$.  Similar computations reveal that  
$MU_i \cong \Z_2 \text{  for all }i.$
  The fact that $KO_i(C\sp*\sr(\Lambda, \gamma)) = 0$ if $i = 6, 7$ implies that $MO_i :=  \rm{image}~\eta_{i-1} \colon KO_{i-1}(C\sr^*(\Lambda,  \gamma)) \rightarrow KO_{i}(C\sr^*(\Lambda,  \gamma)) $ is zero for $i = 0,6,7$. 
  
 The long exact sequence  \eqref{eq:core-exact-seq} implies that $r_6'$ and $c_1'$ are both isomorphisms (thus $MO_1 = MO_5 = \Z_2$) and that $r'_5: MU_5 \to MO_3$ and $r'_6: MU_6 \to MO_4$ must be injective.
From these observations, we obtain the following two segments of sequence \eqref{eq:core-exact-seq}:
  \begin{align}
  0  & \to \Z_2{ \to} MO_4  \xrightarrow{\eta_4'} \Z_2\xrightarrow{c'_5} \Z_2 \xrightarrow{r'_4} MO_2 \to MO_3 \to \Z_2 \to 0 \label{eq:first}\\
  0 & \to \Z_2 \to MO_3 \to MO_4 \xrightarrow{c'_4} \Z_2 \xrightarrow{r'_3} \Z_2 \xrightarrow{\eta'_1}MO_2 \to \Z_2 \to 0
  \label{eq:second}
  \end{align}
  
  Note first that $\eta_4'$ must either be the zero map or be onto.  In the first case, since $r_6'$ is injective, we have $MO_4 = \Z_2$, and $c_5'$ must also be injective (hence an isomorphism). We therefore have $r'_4 = 0$, so Equation \eqref{eq:first}  becomes 
  \[ 0 \to MO_2 \stackrel{\eta_2'}{\to} MO_3 \stackrel{c'_3}{\to} \Z_2 \to 0.\]
  If $\eta_4'$ is onto, then $c_5'$ must be the zero map.  Consequently, $MO_4 = \Z_2^2$ and  Equation \eqref{eq:first} becomes
  \begin{equation}
  \label{eq:first-revised}
  0 \to \Z_2 \stackrel{r'_4}{\to} MO_2 \to MO_3 \to \Z_2 \to 0.\end{equation}

  Similarly,  $\eta'_1$ must be either injective, or the zero map.  In the first case, the fact that each $MO_i$ group is 2-torsion implies that $MO_2 = \Z_2^2$.  Moreover,  $r'_3$ must be the zero map, so Equation \eqref{eq:second} becomes 
  \begin{equation}
  \label{eq:second-revised} 0 \to \Z_2 \to MO_3 \stackrel{\eta'_3}{\to} MO_4 \stackrel{c'_4}{ \to} \Z_2 \to 0.
  \end{equation}
 If $\eta_1' = 0$ then $MO_2 = \Z_2$,  $r_3' = 1$, and $c'_4 =0$, so Equation \eqref{eq:second} becomes 
  \[ 0 \to \Z_2 \to MO_3 {\to} MO_4 \to 0.\]

  Thus, if $\eta_1' = 0$ and $\eta_4' = 0$, the fact that each $MO_i$ group is 2-torsion implies that   $MO_3 = \Z_2^2$.  If $\eta_1' = 0$ and $\eta_4'$ is onto, Equation \eqref{eq:first} becomes 
  \[ 0 \to \Z_2 \to \Z_2 \to MO_3 \to \Z_2 \to 0,\]
  which forces $MO_3 = \Z_2$.  However, Equation \eqref{eq:second} then implies that $MO_4 = 0$, contradicting the fact that (as we observed above) in this case we have $MO_4 = \Z_2^2$. 
  
  If $\eta_1'$ is injective and $\eta_4'=0$, so that $MO_4 = \Z_2$ and $c'_4 = 1$, we must have $\eta_3' = 0$ and hence $MO_3 = \Z_2$.  In other words, $c'_3 = 1$ and $\eta_2' = 0$.  This  forces $MO_2 = 0$, which contradicts the fact that if $\eta_1'$ is injective we have $MO_2 = \Z_2^2$. 
  
  Finally, suppose $\eta_1'$ is injective and $\eta_4'$ is onto, so that $MO_4 = \Z_2^2 = MO_2$. We conclude from Equations \eqref{eq:first-revised} and \eqref{eq:second-revised} that $MO_3 = \Z_2^2$.
  
  In other words, the $MO_i$ groups are given by the first of the tables below if $\eta_4' = \eta_1' = 0$ and by the second if $\eta_4'$ is onto and $\eta_1'$ is injective; no other options are possible.
\[ \begin{array}{|c|c|c|c|c|c|c|c|c|c|}  
\hline  \hline 
 & \makebox[1cm][c]{0} & \makebox[1cm][c]{1} & 
\makebox[1cm][c]{2} & \makebox[1cm][c]{3} 
& \makebox[1cm][c]{4} & \makebox[1cm][c]{5} 
& \makebox[1cm][c]{6} & \makebox[1cm][c]{7} \\
\hline  \hline
MO_i
& 0 & \Z_2 & \Z_2 & \Z_2^2 & \Z_2 & \Z_2 & 0 &  0  \\
\hline  
\hline
\end{array} \]

\[\begin{array}{|c|c|c|c|c|c|c|c|c|c|}  
\hline  \hline 
 & \makebox[1cm][c]{0} & \makebox[1cm][c]{1} & 
\makebox[1cm][c]{2} & \makebox[1cm][c]{3} 
& \makebox[1cm][c]{4} & \makebox[1cm][c]{5} 
& \makebox[1cm][c]{6} & \makebox[1cm][c]{7} \\
\hline  \hline
MO_i
& 0 & \Z_2 & \Z_2^2 & \Z_2^2 & \Z_2^2 & \Z_2 & 0 &  0  \\
\hline  
\hline
\end{array} \]

As we noted at the beginning of this section, the core of the $\CR$-module $S_g$ coincides with the first table above, and the core of the $\CR$-module $T_g$ coincides with the second. Therefore, by \cite[Theorem~4.2.1]{hewitt},
$K\crr(C\sp*(\Lambda,  \gamma))$ is either isomorphic to $S_g$ or to $T_g$.

Comparing the cardinality of $(S_g)^O_2$, $(T_g)^O_2$, and the two options 
for the  $E^\infty$ page of the spectral sequence converging to $KO_*(C\sp*_\R(\Lambda; \gamma))$, we see that 
when $d^2_{(2,1)} \neq 0$ we have
$K\crr(C\sp*(\Lambda,  \gamma)) \cong S_g$ 
and when $d^2_{(2,1)} = 0$ we have
$K\crr(C\sp*(\Lambda,  \gamma)) \cong T_g$. 
\end{proof}

\subsection{A 3-vertex rank-2 graph} \label{SectionExample2}
In this section, we consider a family of rank-2 graph $\Lambda$ with three vertices and with the following adjacency matrices  
\[M_1 = M_2 = \begin{pmatrix} 1 & 1 & 1 \\ 1 & 0 & n-1 \\ 1 & n-1 & 0 \end{pmatrix} \;  \] for $n \geq 2$.
We also consider an involution $ \gamma$ that swaps the second and third vertices. 
(By comparison, a rank-1 graph with involution and with the same adjacency matrix was considered in Example 6.2 in \cite{boersema-MJM}.)
We do not specify the factorization rules for $\Lambda$, since they do not affect our $K$-theory calculations. They may be any factorization rules that are consistent with the involution $ \gamma$.
We consider the real $C\sp*$-algebra $C\sr^*(\Lambda,  \gamma)$.

\begin{prop} The $\CR$ $K$-theory $K\crr(C\sr^*(\Lambda,  \gamma))$ is isomorphic to one of two $\CR$-modules, $P_{2n}$ or $Q_{2n}$.
\end{prop}

The groups of the $\CR$-modules $P_{2n}$ and $Q_{2n}$ are given by the following tables. Again, we only record the groups, not the natural transformations, as these are completely determined by the given groups.  The structure of $Q_{2n}$ differs somewhat depending on $n$ being even or odd.

\[ \begin{array}{|c|c|c|c|c|c|c|c|c|c|}  
\hline  \hline 
& \makebox[1cm][c]{0} & \makebox[1cm][c]{1} & 
\makebox[1cm][c]{2} & \makebox[1cm][c]{3} 
& \makebox[1cm][c]{4} & \makebox[1cm][c]{5} 
& \makebox[1cm][c]{6} & \makebox[1cm][c]{7} \\
\hline  \hline
(P_{2n})_i\po
& \Z_2 & \Z_2^2 & \Z_{4n} \oplus \Z_2 & \Z_{4n} \oplus \Z_2 & \Z_2^2  & \Z_2 & \Z_n & \Z_n  \\
\hline  
(P_{2n})_i\pu
& \Z_{2n} & \Z_{2n} & \Z_{2n} & \Z_{2n} & \Z_{2n} & \Z_{2n} & \Z_{2n} & \Z_{2n}  \\
\hline  
\hline
\end{array} \]

\[ \begin{array}{|c|c|c|c|c|c|c|c|c|}  
\hline  \hline 
\text{($n$ even)} & \makebox[1cm][c]{0} & \makebox[1cm][c]{1} 
&\makebox[1cm][c]{2} & \makebox[1cm][c]{3} 
& \makebox[1cm][c]{4} & \makebox[1cm][c]{5} 
& \makebox[1cm][c]{6} & \makebox[1cm][c]{7} \\
\hline  \hline
 (Q_{2n})_i\po
& \Z_2 & \Z_2^2 & \Z_2 \oplus \Z_{2n} & \Z_2 \oplus \Z_{2n} & \Z_2^2  & \Z_2 & \Z_n & \Z_n  \\
\hline  
 (Q_{2n})_i\pu
& \Z_{2n} & \Z_{2n} & \Z_{2n} & \Z_{2n} & \Z_{2n} & \Z_{2n} & \Z_{2n} & \Z_{2n}  \\
\hline  
\hline
\end{array} \]

\[ \begin{array}{|c|c|c|c|c|c|c|c|c|c|}  
\hline  \hline 
\text{($n$ odd)} & \makebox[1cm][c]{0} & \makebox[1cm][c]{1} & 
\makebox[1cm][c]{2} & \makebox[1cm][c]{3} 
& \makebox[1cm][c]{4} & \makebox[1cm][c]{5} 
& \makebox[1cm][c]{6} & \makebox[1cm][c]{7} \\
\hline  \hline
 (Q_{2n})_i\po
& \Z_2 & \Z_4 & \Z_2^2 \oplus \Z_n & \Z_2^2 \oplus \Z_{n} & \Z_4  & \Z_2 & \Z_n & \Z_n  \\
\hline  
 (Q_{2n})_i\pu
& \Z_{2n} & \Z_{2n} & \Z_{2n} & \Z_{2n} & \Z_{2n} & \Z_{2n} & \Z_{2n} & \Z_{2n}  \\
\hline  
\hline
\end{array} \]

The cores of these $\CR$-modules include the groups below: 

\[ \begin{array}{|c|c|c|c|c|c|c|c|c|c|}  
\hline  \hline 
\text{Core of~} P_{2n}   & \makebox[1cm][c]{0} & \makebox[1cm][c]{1} & 
\makebox[1cm][c]{2} & \makebox[1cm][c]{3} 
& \makebox[1cm][c]{4} & \makebox[1cm][c]{5} 
& \makebox[1cm][c]{6} & \makebox[1cm][c]{7} \\
\hline  \hline
MO_i
& 0 & \Z_2 & \Z_2^2 & \Z_2^2 & \Z_2^2 & \Z_2 & 0 &  0  \\
MU_i 
& \Z_2 & \Z_2 & \Z_2 & \Z_2 & \Z_2 & \Z_2 & \Z_2 & \Z_2  \\
\hline  
\hline
\end{array} \]

\[ \begin{array}{|c|c|c|c|c|c|c|c|c|c|}  
\hline  \hline 
\text{Core of~} Q_{2n}     & \makebox[1cm][c]{0} & \makebox[1cm][c]{1} & 
\makebox[1cm][c]{2} & \makebox[1cm][c]{3} 
& \makebox[1cm][c]{4} & \makebox[1cm][c]{5} 
& \makebox[1cm][c]{6} & \makebox[1cm][c]{7} \\
\hline  \hline
MO_i
& 0 & \Z_2 & \Z_2 & \Z_2^2 & \Z_2 & \Z_2 & 0 &  0  \\
MU_i 
& \Z_2 & \Z_2 & \Z_2 & \Z_2 & \Z_2 & \Z_2 & \Z_2 & \Z_2  \\
\hline  
\hline
\end{array} \]

We note that there is a $\CR$-module isomoprhism 
\[  P_{2n} \cong \Sigma^{-2} K\crr( \mathcal{E}_{2n+1}) \oplus \Sigma^{-3} K\crr( \mathcal{E}_{2n+1}) \; , \] 
where $\mathcal{E}_{2n+1}$  is the exotic Cuntz algebra described in Section~11 of \cite{brs}.
{However, to our knowledge, the $\CR$-modules $Q_{2n}$ have not previously been discussed in the literature.}

\begin{proof}

We will develop the chain complex, and subsequent spectral sequence, as in Theorem~\ref{k=2--chaincomplex} to compute 
$K\crr(C\sr^*(\Lambda,  \gamma))$.
The chain complex is
\[ 0 \rightarrow A \xrightarrow{\partial_2} A^2 \xrightarrow{\partial_1} A \rightarrow 0\]
where $A = K\crr(\R) \oplus K\crr(\C)$.
Using Theorem \ref{k=2--chaincomplex}
and the fact that for $i = 1,2$ we have 
\[ \rho^i = B =  I_3 - M_i  = \begin{pmatrix} 0 & -1 & -1 \\ -1 & 1 &1- n \\ -1 & 1-n & 1 \end{pmatrix}  \; ,\]
we can analyze the groups and maps of this chain complex in each grading, complex and real parts, as below.

\begin{align*}
~& \text{ \underline{complex part:}} \\
\text{ degree $0$} : \quad & 0 \rightarrow \Z^3 
	\xrightarrow{\begin{pmatrix}  -B \\ B \end{pmatrix} }   \Z^6 
	\xrightarrow{\begin{pmatrix} B & B \end{pmatrix}}  \Z^3 \rightarrow 0 \\
~\\
~& \text{ \underline{ real part:}} \\
\text{ degree $0$} : \quad & 0 \rightarrow \Z^2 
	\xrightarrow{\begin{pmatrix} 0 & 2 \\ 1 & n-2 \\ 0 & -2 \\ -1 & 2-n \end{pmatrix}}   \Z^4 
	\xrightarrow{\begin{pmatrix} 0 & -2 & 0 & -2 \\ -1 & 2-n & -1 & 2-n \end{pmatrix}}  \Z^2 \rightarrow 0 \\
\text{ degree $1$}: \quad & 0 \rightarrow \Z_2 \xrightarrow{\begin{pmatrix} 0 \\ 0 \end{pmatrix}} 
	\Z_2^2 \xrightarrow{\begin{pmatrix} 0 & 0 \end{pmatrix}} \Z_2 \rightarrow 0
	 \\ 
\text{ degree $2:$} \quad & 0 \rightarrow 
	( \Z_2 \oplus \Z) 
	\xrightarrow{\begin{pmatrix} 0 & 1 \\ 0 & -n \\ 0 & -1 \\ 0 & n \end{pmatrix}}
	(\Z_2 \oplus \Z)^2 
	\xrightarrow{\begin{pmatrix} 0 & -1 &0 & -1 \\ 0 & n & 0 & n  \end{pmatrix}}
	(\Z_2 \oplus \Z)
	\rightarrow 0 \\
\text{ degree $4$} : \quad & 0 \rightarrow \Z^2 
	\xrightarrow{\begin{pmatrix} 0 & 1 \\ 2 & n-2 \\ 0 & -1 \\ -2 & 2-n \end{pmatrix}}   \Z^4 
	\xrightarrow{\begin{pmatrix} 0 & -1 & 0 & -1 \\ -2 & 2-n & -2 & 2-n \end{pmatrix}}  \Z^2 \rightarrow 0 \\
\text{ degree $6$}: \quad  & 0 \rightarrow \Z \xrightarrow{ \begin{pmatrix} -n \\ n \end{pmatrix} } \Z^2 \xrightarrow{ \begin{pmatrix} n & n \end{pmatrix} }  \Z \rightarrow 0 \\
\end{align*} 

The Smith normal form of $B$ is 
\[\snf(B) 
= \begin{pmatrix} 1 & 0 & 0 \\ 0 & 1 &0 \\ 0 & 0 & 2n \end{pmatrix}  \; . \]
From this, it easily follows that 
\begin{align*}
(E_{0,0}^2)\pu &= \text{coker}\, (B \, B) \cong \Z_{2n} \, , \\
(E_{1,0}^2)\pu &= \ker (B \, B) /\text{image}\, \begin{pmatrix}
-B \\ B 
\end{pmatrix} \cong \Z_{2n} \, , \\
(E_{2,0}^2)\pu &= \ker  \begin{pmatrix}
-B \\ B \end{pmatrix} = 0
\end{align*}

For the real part, we work out the homology of the exact sequences associated to the ``real part'' above to obtain $(E^2_{p,q})\po$ (or simply $E^2_{p,q}$, as we will denote it when it is clear). We will walk through the details of this for the first three rows and leave the rest to the reader. The Smith normal form of the matrix $ \left( \begin{smallmatrix} 0 & 2 \\ 1 & n-2 \end{smallmatrix} \right)$ is $ \left( \begin{smallmatrix} 1 & 0 \\ 0 & 2 \end{smallmatrix} \right)$ for all $n$. It follows that in the real part in degree 0 (that is, when $q = 0$) we have $E^2_{0,0} = E^2_{1,0} = \Z_2$, and $E^2_{2,0} = 0$. 

When $q = 1$ we have $\partial_1 = \partial_2 = 0$, so it immediately follows that $E_{0,1}^2 = \Z_2, E_{1,1}^2 = \Z_2^2,$ and $E_{1,2}^2 = \Z_2$.

For the next row (when $q = 2$), we first observe that 
\[\text{image}\, \begin{pmatrix} 0 & -1 &0 & -1 \\ 0 & n & 0 & n  \end{pmatrix} = \{ ( [1], (2k+1) n) : k \in \Z\} \cup \{ ([0], 2kn) : k\in \Z)\} \subseteq \Z_2 \oplus \Z.\]
In particular, the sum $([1], 1)$ of the two generators $([1], 0)$ and $([0], 1)$ of $\Z_2 \oplus \Z$ lies in $\im \begin{pmatrix} 0 & -1 &0 & -1 \\ 0 & n & 0 & n  \end{pmatrix}$.
Consequently,  
\[ E_{0,2}^2 = \coker \begin{pmatrix} 0 & -1 &0 & -1 \\ 0 & n & 0 & n  \end{pmatrix} = \langle [([0], 1)] \rangle = \Z_{2n}.\]
  
To show that $E^2_{1,2} = \Z_2 \oplus \Z_{2n}$, we note that
\[ \ker \partial_1 = \ker \begin{pmatrix} 0 & -1 & 0 & -1 \\ 0 & n & 0 & n \end{pmatrix}
= \{ ( x, y, z, -y) \mid x,z \in \Z_2, y \in \Z \} \]
while
\[ \im \partial_2
	= \{ ([x], -nx, [x], nx) \mid x \in \Z \}  \; .\]
{Consequently, $E^2_{1,2} = \ker \partial_1/ \im \partial_2 = \langle [(0,0,1,0)], [(0,1,0,-1)]\rangle \cong \Z_2 \oplus \Z_{2n}$ because $(0, n, 0, -n) \not\in \im \partial_2$ but $(0,2n, 0, -2n)$ is.}

Finally, note that 
\[ E^2_{0,2} = \ker \partial_2 = \ker \begin{pmatrix} 0 & 1 \\ 0 & -n \\ 0 & 1 \\ 0 & n \end{pmatrix}
	= \{ (x, 0) \mid x \in \Z_2 \} = \Z_2 \, .\]

The $E^2_{p,q}$ groups of the spectral sequence  converging to $KO_*(C\sp*\sr(\Lambda, \gamma))$ are shown on the left below. 

From this and similar calculations for $3 \leq q \leq 7$ we obtain the $E^2$ page of the spectral sequence as shown.

\begin{gather*}
\text{ \underline{ $E^2_{p,q} $ }}\\
\def\vvline{\hfil\kern\arraycolsep\vline\kern-\arraycolsep\hfilneg}
\begin{array}{ cccc }
 \multicolumn
{4}{c}{ \underline{ \text{real part}}} \\
\vspace{.25cm} \\
~ \vvline& \hspace{.3cm} \vdots \hspace{.3cm} &\hspace{.3cm} \vdots \hspace{.3cm} & \hspace{.3cm} \vdots \hspace{.3cm} \\
7  \vvline & 0 & 0 & 0  \\
6   \vvline  &  \Z_n &  \Z_n  & 0 \\
5   \vvline  &  0    &  0   &  0    \\
4   \vvline  &  \Z_{2} & \Z_{2}  & 0   \\
3    \vvline &  0  & 0 &  0   \\
 2   \vvline  &  \Z_{2n}  &  \Z_2 \oplus \Z_{2n}   & \Z_2  \\
 1   \vvline  &   \Z_2  &  \Z_2^2   & \Z_2  \\
0 \vvline & \Z_2 & \Z_2 & 0 \\ \hline
~ \vvline & 0 & 1 & 2
\end{array}
\hspace{2cm}
\begin{array}{ cccc }
 \multicolumn
{4}{c}{ \underline{ \text{complex part}}} \\
\vspace{.25cm} \\
~ \vvline & \hspace{.3cm} \vdots \hspace{.3cm} &\hspace{.3cm} \vdots \hspace{.3cm} & \hspace{.3cm} \vdots \hspace{.3cm} \\
7  \vvline & 0 & 0 & 0  \\
6   \vvline  &  \Z_{2n} &  \Z_{2n}  & 0 \\
5   \vvline  &  0    &  0   &  0    \\
4   \vvline  &  \Z_{2n} &  \Z_{2n}  & 0 \\
3    \vvline &  0  & 0 &  0   \\
2   \vvline  &  \Z_{2n} &  \Z_{2n}  & 0 \\
 1   \vvline  & 0    &  0   &  0    \\
0   \vvline  &  \Z_{2n} &  \Z_{2n}  & 0 \\ \hline
~ \vvline & 0 & 1 & 2
\end{array}
\end{gather*}

The map $d^2$ is forced to be 0 everywhere except possibly the map $d^2_{(2,1)} \colon \Z_2 \rightarrow \Z_{2n}$ from degree $(2,1)$ to degree $(0,2)$ in the real case.

The complex spectral sequence (on the right) tells us $KU_i(C\sr^* (\Lambda,  \gamma)) \cong \Z_{2n}$ for all $i$.
Thus $KU_0(C\sp*\sr(\Lambda,\gamma)) = KU_1(C\sp*\sr(\Lambda, \gamma)) \cong \Z_{2n}$. 
It follows that the complex $C\sp*$-algebra $C\sp*(\Lambda)$ is $KK$-equivalent to 
$\mathcal{O}_{2n+1} \otimes \mathcal{O}_{2n+1}$.
 
With more work, we can identify the maps $\psi$, by tracing the elements $KU_*(C\sr^* (\Lambda,  \gamma))$ as they arise from the chain complex through the spectral sequence. In the $i = 0$ case we have that $KU_0(C\sr^* (\Lambda,  \gamma))$ is isomorphic to $A\pu_0/\im B = \Z^3/\im B$, and the generator of $KU_0(C\sr^* (\Lambda,  \gamma)) \cong \Z_{2n}$ is represented by the element $(0, 1, 0)$ which is equivalent to the element $(0,0,-1)$ (since $(0, -1, -1)$ is in the image of $B$). 
Similarly, one computes that 
\[\ker ( B \ B) = \{ (x, y, z, -x, -y, -z) \in \Z^6\},\]
 so $[(0,1,0,0,-1,0)] = [(0,0,-1, 0, 0, 1)]$ generates $KU_1(C\sp*\sr(\Lambda, \gamma)) = \ker (B \ B)/\text{image}\begin{pmatrix}
-B \\ B
\end{pmatrix}$.

As $A_0\pu = KU_0(\R) \oplus KU_0(\C)$, the fact that $(\psi_\C)_0 = \begin{pmatrix}
0& 1 \\ 1 & 0 \end{pmatrix}$
on $KU_0(\C) = \Z^2$ implies that
$(\psi_A)_0(x,y,z) = (x,z,y)$. Thus, the involution $\psi_0$ on $KU_0(C\sr^* (\Lambda,  \gamma)) \cong A\pu_0/\im B \cong \Z^3/\im B$ induced by $\psi_A$ 
satisfies $\psi_0([0,1,0]) = [0, 0, 1].$ 
It follows that $\psi_0$ is given by multiplication by $-1$ in $KU_0(C\sr^*(\Lambda,  \gamma)) = \Z_{2n}$. {A similar analysis} also shows that $\psi = -1$  in $KU_1( C\sr^*(\Lambda,  \gamma)) = \Z_{2n}$. Using the fact that $\psi$ anticommutes with the Bott isomorphism, that is 
$\psi \beta = - \beta \psi$, we find that $\psi_i = -1$ for $i = 0,1,4,5$ and $\psi_i = 1$ for $i = 2,3,6,7$.

In addition to $\psi$, it would be possible to compute the action of most of the natural transformations $r, c, \eta$ in this way, based on the corresponding actions in $A$. Alternatively, once we have computed a few of these natural transformations,  
we can complete the calculation of $K\crr(C\sp*\sr(\Lambda, \gamma))$ using the long exact sequence \eqref{eq:CR-LES}  and the core exact sequence \eqref{eq:core-exact-seq}.

 From the $E^2$ page of the spectral sequence for $KO_*(C\sp*\sr(\Lambda, \gamma))$ we  see immediately that  $KO_0(C\sr^*(\Lambda,  \gamma)) \cong \Z_2$ and $KO_1(C\sr^*(\Lambda,  \gamma))$ is either isomorphic to $\Z_4$ or to $\Z_2^2$. 
Less immediately, we  also find that $\eta_0$ and $\eta_1$ are non-trivial.

To see that $\eta_0$ is non-trivial, observe that  
\[ KO_0(C\sr^*(\Lambda,  \gamma)) \cong (KO_0(\R) \oplus KO_0(\C) )/\im \begin{pmatrix}
0 & -2 & 0 & -2 \\ -1 & 2-n & -1 & 2-n
\end{pmatrix} \] 
is generated by $[(1,0)]$. 
As $(\eta_\R)_0 ([1]) \in KO_1(\R)$ is the non-trivial element of $(A_1)\po = \Z_2$, which is the $E^2_{0,1}$ group of the spectral sequence converging to $KO_*(C\sp*\sr(\Lambda, \gamma))$, we conclude that $\eta_0([1,0])$ is a non-trivial element of $KO_1( C\sr^*(\Lambda,  \gamma))$. Therefore $\eta_0 \neq 0$. In addition, since $(\eta_\R)_1: \Z_2 \to \Z_2$ 
is non-trivial and the image in $E^2_{2,0} = \coker \begin{pmatrix}
0 & -1 & 0 & -1 \\ 0 & n & 0 & n 
\end{pmatrix} $ of the generator $([1], 0)$ of $\Z_2 \subseteq \Z_2 \oplus \Z$ is nontrivial, we conclude that $\eta_1: KO_1(C\sr^*(\Lambda,  \gamma)) \to KO_2(C\sr^*(\Lambda, \gamma))$  is also nontrivial.

Now we claim that $r_i \colon KU_i(C\sr^*(\Lambda,  \gamma)) \rightarrow KO_i(C\sr^*(\Lambda,  \gamma))$ is surjective for $i = 5,6,7$.  
In degree 6, since $A = K\crr(\R) \oplus K\crr(\C)$ and $(r_\C)_6 = \begin{pmatrix}
-1 & 1
\end{pmatrix}$, we conclude that $(r_A)_6 \colon \Z^3 \rightarrow \Z$ on $KO_*(A)$ is given by $(x,y,z) \mapsto z-y$. Now, recall that the generator of $KU_6(C\sr^*(\Lambda,  \gamma)) = \coker\begin{pmatrix}
B & B
\end{pmatrix}= \Z_{2n}$ is represented by $(0,1,0)$.  Thus, $r_6: KU_6(C\sr^*(\Lambda,  \gamma)) \to KO_6(C\sr^*(\Lambda, \gamma))$ satisfies $r_6([0,1,0]) = [-1] \in \Z_n \cong KO_6(C\sp*\sr(\Lambda, \gamma))$.    Thus $r_6$ is onto.

To see that $r_5$ is onto, recall that for $i$ odd
\[KU_i(C\sp*\sr(\Lambda,\gamma)) = \ker (B \ B ) /\text{image}\begin{pmatrix}
-B \\ B
\end{pmatrix} = \Z_{2n} \]
is generated by $g = [(0,1,0,0,-1,0)]$.
Also, for $i = 5$ we find (referring to the degree 4 part of the chain complex)
\[KO_5(C\sp*\sr(\Lambda,\gamma)) = \ker (\partial_1)_4 / \im (\partial_2)_4 = \Z_2 \]
and the non-trivial element can be determined to be represented by $h = [(0, 1, 0, -1)]$.
The map $r_5 \colon KU_5(C\sp*\sr(\Lambda,\gamma)) \rightarrow KO_5(C\sp*\sr(\Lambda,\gamma))$ is then induced by the map
$(r_A)_4: \Z^6 \to \Z^4,$ which is given by the formula
\[ (x,y,z,u,v,w) \mapsto (x, y+z, u, v+w) \; .\]
since $(r\sc)_4 = \begin{pmatrix}
1 & 1
\end{pmatrix}: \Z^2 \to \Z$.
Thus $r_5(g) = [(0, 1, 0, -1)]$, so $r_5$ is surjective.

Following a similar argument, we can show that $r_7 \colon KU_7(C\sp*\sr(\Lambda,\gamma)) \rightarrow KO_7(C\sp*\sr(\Lambda,\gamma))$ is onto. As $KU_7(C\sp*\sr(\Lambda, \gamma)) = {\ker \begin{pmatrix} B & B
\end{pmatrix} }/{\im\begin{pmatrix}
-B \\B
\end{pmatrix}}$,   $r_7$ is induced from $(r_{A^2})_6: \Z^6 \to \Z^2$, and $(r_{A^2})_6(x,y,z,u,v,w) = (z-y, w-v)$.  Therefore, using the generator $g$ of $KU_7(C\sp*\sr(\Lambda,\gamma))$ identified above, 
\[ r_6(g) = [(1, -1)] \in {\ker \begin{pmatrix}
n & n 
\end{pmatrix}}\left/{\im \begin{pmatrix}
-n \\ n
\end{pmatrix}}\right.= KO_6(C\sp*\sr(\Lambda, \gamma)) .\]
As $[(1,-1)]$ generates $KO_6(C\sp*\sr(\Lambda,\gamma))$, we conclude that $r_7$
 is also surjective.

Since $r_i$ is surjective for $i = 5,6,7$ it immediately follows from \eqref{eq:natural-transformations} that $\eta_i = 0$ for $i =5,6,7$.  

Now, we turn to the core of the exact sequence. Given that we've already computed $KU_i(C\sr^*(\Lambda,  \gamma)) = \Z_{2n}$ and $\psi_i = \pm 1$, Equation \eqref{eq:core-gps} implies that 
$MU_i(C\sr^*(\Lambda,  \gamma)) = {\Z_{2}}$ for all $i$. The fact that $\eta_i = 0$ for $5 \leq i \leq 7$ implies that 
$MO_i(C\sr^*(\Lambda,  \gamma)) = 0$ for $i = 0,6,7$. From this information, the long exact sequence \eqref{eq:core-exact-seq} relating $MO_i$ and $MU_i$ can be used to determine that $MO_*$ must be one of the following (using the same argument as used in the previous section). 

\[ \begin{array}{|c|c|c|c|c|c|c|c|c|c|}  
\hline  \hline 
  & \makebox[1cm][c]{0} & \makebox[1cm][c]{1} & 
\makebox[1cm][c]{2} & \makebox[1cm][c]{3} 
& \makebox[1cm][c]{4} & \makebox[1cm][c]{5} 
& \makebox[1cm][c]{6} & \makebox[1cm][c]{7} \\
\hline  \hline
MO_i
& 0 & \Z_2 & \Z_2^2 & \Z_2^2 & \Z_2^2 & \Z_2 & 0 &  0  \\
\hline  
\hline
\end{array} \]
or
\[ \begin{array}{|c|c|c|c|c|c|c|c|c|c|}  
\hline  \hline 
  & \makebox[1cm][c]{0} & \makebox[1cm][c]{1} & 
\makebox[1cm][c]{2} & \makebox[1cm][c]{3} 
& \makebox[1cm][c]{4} & \makebox[1cm][c]{5} 
& \makebox[1cm][c]{6} & \makebox[1cm][c]{7} \\
\hline  \hline
MO_i
& 0 & \Z_2 & \Z_2 & \Z_2^2 & \Z_2 & \Z_2 & 0 &  0  \\
\hline  
\hline
\end{array} \]

The former possibility coincides with the core of $P_{2n}$.  Comparing the group $(P_{2n})_3^O$ with the groups $(E^O)^3_{p,q}$ of the spectral sequence converging to $KO_*(C\sp*\sr(\Lambda, \gamma))$ for $p+q=3$, we see that this possibility coincides with the case when $d^2_{2,1} = 0$. Therefore, in that case
$K\crr(C\sr^*(\Lambda,  \gamma)) = P_{2n}$. The latter case occurs when $d^2_{2,1} \neq 0$ and yields 
$K\crr(C\sr^*(\Lambda,  \gamma)) = Q_{2n}$. 
\end{proof}

If the factorization rules of $\Lambda$ are such that $C\sp*(\Lambda)$ is simple and purely infinite, then the complex \calg ~$C\sp*(\Lambda)$ is isomorphic to a matrix algebra over $\mathcal{O}_{2n+1} \otimes \mathcal{O}_{2n+1}$. 
A little more work will be necessary to track the class of $[1]$ in $KU_0(C\sp*\sr(\Lambda, \gamma))$ to determine the value of $k$ in the isomorphism
$C\sp*(\Lambda) \cong M_k(\mathcal{O}_{2n+1} \otimes \mathcal{O}_{2n+1})$.
The real \calg ~$C\sp*\sr(\Lambda, \gamma)$ is then a real structure of $M_k(\mathcal{O}_{2n+1} \otimes \mathcal{O}_{2n+1})$ but is not isomorphic (nor stably isomorphic) to
$\mathcal{O}\pr_{2n+1} \otimes\sr \mathcal{O}\pr_{2n+1}$ or any other tensor product of real Cuntz algebras. This follows, for example, from the fact that $KO_7(C\sp*\sr(\Lambda, \gamma)) \cong \Z_2$, but $KO_7(-)$ is trivial for any tensor product of real Cuntz algebras. Furthermore, $C\sp*\sr(\Lambda, \gamma)$ is not isomorphic to any real $C\sp*$-algebra arising from a rank-1 graph with involution, as $KO_7(-)$ is torsion-free for such a $C\sp*$-algebra by Corollary~4.3 of \cite{boersema-MJM}.

\subsection{Another 3-vertex 2-graph} \label{SectionExample3}
In this section we examine another 2-graph for which the two adjacency matrices are not the same.
Fix an integer $n \geq 2$.  Let $\Lambda$ be a rank-2 graph with three vertices with adjacency matrices
\[ M_1 =  \begin{pmatrix} 1 & 1 & 1 \\ 1 & 0 & n-1 \\ 1 & n-1 & 0 \end{pmatrix} \quad \text{and} \quad 
	M_2 = \begin{pmatrix} 1 & 1 & 1 \\ 1 &  n-1 & 0 \\ 1 & 0 & n-1  \end{pmatrix} \]
 and with an involution $\gamma$ that swaps the second and third vertices.

\begin{prop}
Fix an integer $n \geq 2$.
If $n$ is odd, then $K\crr( C\sp* (\Lambda,  \gamma)) \cong S_2$ or $T_2$.
If $n$ is even, then
\[
K\crr( C\sp* (\Lambda,  \gamma))  \cong \Sigma K\crr( \mathcal{O}\pr_3) \oplus \Sigma^{-2} K\crr( \mathcal{O}\pr_3) 
 \cong \Sigma^{-1} K\crr( \mathcal{E}_3) \oplus  \Sigma^4 K\crr( \mathcal{E}_3) \; .  \]
\end{prop}

We find it intriguing that for $n$ even, the choice of $n$ has no impact on the $\CR$ $K$-theory groups of the real $C\sp*$-algebra.
For all odd integers $n$, there are only two possible $\CR$-modules that can be realized by $K\crr(C\sp*(\Lambda, \gamma))$.

\begin{proof}
Again, we use Theorem \ref{k=2--chaincomplex}. The chain complex $\mathcal A^{(0)}$ used to build the spectral sequence to compute $K\crr( C\sp*(\Lambda,  \gamma))$ has the following components,
based on the matrices below:
\[ \rho^1 = B_1 = I_3 - M_1  =  \begin{pmatrix} 0 & -1 & -1 \\ -1 & 1 & 1-n \\ -1 & 1-n & 1 \end{pmatrix} \quad \text{and} \quad 
	\rho^2 = B_2 = I_3 - M_2 = \begin{pmatrix} 0 & -1 & -1 \\ -1 &  2-n & 0 \\ - 1 & 0 & 2-n  \end{pmatrix} .\]

\begin{align*}
~& \text{ \underline{complex part:}} \\
0 : \quad & 0 \rightarrow \Z^3 
	\xrightarrow{\begin{pmatrix}  -B_2 \\ B_1 \end{pmatrix} }   \Z^6 
	\xrightarrow{\begin{pmatrix} B_1 & B_2 \end{pmatrix}}  \Z^3 \rightarrow 0 \\
~\\
~& \text{ \underline{ real part:}} \\
0 : \quad & 0 \rightarrow \Z^2 
	\xrightarrow{\begin{pmatrix} 0 & 2 \\ 1 & n-2 \\ 0 & -2 \\ -1 & 2-n \end{pmatrix}}   \Z^4 
	\xrightarrow{\begin{pmatrix} 0 & -2 & 0 & -2 \\ -1 & 2-n & -1 & 2-n \end{pmatrix}}  \Z^2 \rightarrow 0 \\
1: \quad & 0 \rightarrow \Z_2 \xrightarrow{\begin{pmatrix} 0 \\ 0 \end{pmatrix}} 
	\Z_2^2 \xrightarrow{\begin{pmatrix} 0 & 0 \end{pmatrix}} \Z_2 \rightarrow 0
	 \\ 
2: \quad & 0 \rightarrow 
	( \Z_2 \oplus \Z) 
	\xrightarrow{\begin{pmatrix} 0 & 1 \\ 0 & n-2 \\ 0 & -1 \\ 0 & n \end{pmatrix}}
	(\Z_2 \oplus \Z)^2 
	\xrightarrow{\begin{pmatrix} 0 & -1 &0 & -1 \\ 0 & n & 0 & 2-n  \end{pmatrix}}
	(\Z_2 \oplus \Z)
	\rightarrow 0 \\
4 : \quad & 0 \rightarrow \Z^2 
	\xrightarrow{\begin{pmatrix} 0 & 1 \\ 2 & n-2 \\ 0 & -1 \\ -2 & 2-n \end{pmatrix}}   \Z^4 
	\xrightarrow{\begin{pmatrix} 0 & -1 & 0 & -1 \\ -2 & 2-n & -2 & 2-n \end{pmatrix}}  \Z^2 \rightarrow 0 \\
6: \quad  & 0 \rightarrow \Z \xrightarrow{ \begin{pmatrix} n-2 \\ n \end{pmatrix} } \Z^2 \xrightarrow{ \begin{pmatrix} n & 2-n \end{pmatrix} }  \Z \rightarrow 0 \\
\end{align*} 

The Smith normal forms of the matrices in the complex part of the chain complex are 
\[ \snf(B_1, B_2) = \begin{pmatrix} 1 & 0 & 0 & 0 & 0 & 0 \\ 0 & 1 & 0 & 0 & 0 & 0 \\ 0 & 0 & 2 & 0 & 0 & 0 \end{pmatrix}  
\quad \text{and} \quad \snf\left( \begin{matrix} - B_2\\ B_1 \end{matrix} \right)  
	= \begin{pmatrix} 1 & 0 & 0 \\ 0 & 1 &0 \\ 0 & 0 & 2 \\ 0 & 0 & 0 \\ 0 & 0 & 0 \\ 0 & 0 & 0 \end{pmatrix} \]
Then we have 
$\text{coker}\, (B_1 \, B_2) \cong \Z_{2}$,
$\ker (B_1 \, B_2) / \im \begin{pmatrix}-B_2 \\ B_1 \end{pmatrix} \cong \Z_2$, 
and $\ker  \begin{pmatrix} -B_2 \\ B_1  \end{pmatrix} = 0$.
Thus, the $E^2$ groups of the chain complex computing $KU_*(C\sp*\sr(\Lambda,\gamma))$ satisfy 
\[ (E^2_{p,q})^U = \begin{cases}
\Z_2 & p\in \{ 0,1\}, q \text{ even}\\
0 & \text{otherwise.}
\end{cases}\]
Consequently, $KU_0(C\sp*\sr(\Lambda,\gamma)) \cong KU_1(C\sp*\sr(\Lambda, \gamma)) \cong \Z_{2}$.
It follows that the complex $C\sp*$-algebra $C\sp*(\Lambda)$ has the same $K$-theory as
$\mathcal{O}_{3} \otimes \mathcal{O}_{3}$.

From the chain complexes exhibited above, we can compute that the $E^2$ page of the spectral sequence computing 
$KO_*(C\sp*\sr(\Lambda, \gamma))$ is  given by one of the following tables. The left-hand table corresponds to the case where $n$ is odd and the right-hand table corresponds to the case where $n$ is even.  It is a remarkable fact that, even though the matrices $\rho^i$ look very different when $n=2$, the $E^2$ page in this case is the same as for any other even $n$.

\begin{gather*}
\text{ \underline{ real part of $E^2_{p,q} $ }}\\
\def\vvline{\hfil\kern\arraycolsep\vline\kern-\arraycolsep\hfilneg}
\begin{array}{ cccc }
 \multicolumn
{4}{c}{ \underline{ \text{case 1: $n$ is odd}}} \\
\vspace{.25cm} \\
~ \vvline& \hspace{.3cm} \vdots \hspace{.3cm} &\hspace{.3cm} \vdots \hspace{.3cm} & \hspace{.3cm} \vdots \hspace{.3cm} \\
7  \vvline & 0 & 0 & 0  \\
6   \vvline  &  0 & 0 & 0  \\
5   \vvline  &  0    &  0   &  0    \\
4   \vvline  &  \Z_{2} & \Z_{2}  & 0   \\
3    \vvline &  0  & 0 &  0   \\
 2   \vvline  & \Z_2  &  \Z_2^2   & \Z_2  \\
  1   \vvline  &   \Z_2  &  \Z_2^2   & \Z_2  \\
0 \vvline & \Z_2 & \Z_2 & 0 \\ \hline
~ \vvline & 0 & 1 & 2
\end{array}
\hspace{2cm}
\begin{array}{ cccc }
 \multicolumn
{4}{c}{ \underline{ \text{case 2: $n$ is even}}} \\
\vspace{.25cm} \\
~ \vvline & \hspace{.3cm} \vdots \hspace{.3cm} &\hspace{.3cm} \vdots \hspace{.3cm} & \hspace{.3cm} \vdots \hspace{.3cm} \\
7  \vvline & 0 & 0 & 0  \\
6   \vvline  &  \Z_{2} &  \Z_{2}  & 0 \\
5   \vvline  &  0    &  0   &  0    \\
4   \vvline  &  \Z_{2} & \Z_{2}  & 0   \\
3    \vvline &  0  & 0 &  0   \\
 2   \vvline  & \Z_2  &  \Z_2^2   & \Z_2  \\
  1   \vvline  &   \Z_2  &  \Z_2^2   & \Z_2  \\
0 \vvline & \Z_2 & \Z_2 & 0 \\ \hline
~ \vvline & 0 & 1 & 2
\end{array}
\end{gather*}

As an illustration of how this spectral sequence was obtained, we explain the computations of $E^2_{p, 2}$ for $p = 0,1,2$.  We have 

\[E^2_{0,2} = \coker (\partial_1)_2 =
 \coker \begin{pmatrix}
0 & -1 & 0 & -1 \\ 0 & n & 0 & 2-n
\end{pmatrix} =  (\Z_2 \oplus \Z)/G \]
where $G$ is the subgroup of $\Z_2 \oplus \Z)$ generated by $([1], n)$ and by $([0], 2)$.
This in turn is isomorphic to $(\Z_2 \oplus \Z_2)/G'$ where $G'$ is the subgroup of $\Z_2 \oplus \Z_2$ generated by $( [1], [n])$.
Note that in $\Z_2 \oplus \Z_2$, $([1], [n])$ is equal to $([1], [0])$ or $([1], [1])$, depending on whether $n$ is even or odd. In either case, the quotient is isomorphic to $\Z_2$.

To compute $E^2_{1,2} = \ker (\partial_1)_2/\im (\partial_2)_2 $, we first compute that 
\begin{align*}
\ker \begin{pmatrix}
0 & -1 & 0 & -1 \\ 0 & n & 0 & 2-n \end{pmatrix} 
&= \{ ([x], y, [z], w) \mid  y + w \equiv 0 \pmod 2, ~n y =  (n-2) w \} \; \\
&= \{ ([x], k(n-2), [z], kn) \mid  k \in \Z\}.
\end{align*}
Now, observe that 
\[ \im ( \partial_2)_2 = \im \begin{pmatrix}
0 & 1 \\ 0 & n-2 \\ 0 & -1 \\ 0 & n  
\end{pmatrix} = \{ (  [x], (n-2) x, [x], n x ) \mid x \in \Z\}
\subseteq (\Z_2 \oplus \Z)^2 \; . \]
Thus a generic element  in 
$E^2_{1,2}$ can be written as
\[ [ ([x], y, [z], w) ] = [ ([x], k(n-2), [z], kn ) ] = [( [k+x],0, [k + z], 0 )] \; ,  \]
since $([1], n-2, [1], n) \in \im (\partial_2)_2$ and hence
$[(0, n-2, 0, n)] = [( [1], 0, [1], 0 )]$.
It now follows easily that $E^2_{1,2} \cong \Z_2^2$.

Finally, 
\[E^2_{2,2} = \ker (\partial_2)_2 =  \ker \begin{pmatrix}
0 & 1 \\ 0 & n-2 \\ 0 & -1 \\ 0 & n  
\end{pmatrix} = \{ ([x], y) \mid y = 0 \}
\cong \Z_2 \; . \]

Leaving it to the reader to calculate $E^2_{p,q}$ for the remaining values of $p,q$, we turn now to analyze the spectral sequence from the $E^2$ stage.
In the case where $n$ is odd, the spectral sequence is the same as one that we saw in Section~\ref{SectionExample1}, 
so, as there,  we can conclude that either 
$K\crr( C\sp* (\Lambda,  \gamma)) \cong S_2$ (if $d_{(2,1)}^2 \not= 0$) or $K\crr( C\sp* (\Lambda,  \gamma)) \cong T_2$ (if $d^2_{(2,1)}=0$).

In the case where $n$ is even, there is some work to do. Observe first that once again, we have $d^2_{(i,j)} =0 $ for all $(i,j)$ with the possible exception of $d^2_{(2,1)}$. Thus, for all $(i,j) \not\in \{ (0,2), (2,1)\}$ we have $E^2_{ij} = E^\infty_{ij}$.  It follows that 
$KO_i( C\sp* (\Lambda,  \gamma))
	\cong \Z_2$ for $i = 5,6,7,0$, and
that $|KO_i( C\sp* (\Lambda,  \gamma))| = 4$ for $i = 1,4$.
If $d^2_{(2,1)}=0$ then 
$|KO_2 (C\sp* \sr(\Lambda,  \gamma))| = |KO_3(C\sp*\sr (\Lambda,  \gamma))| = 8$, and if $d^2_{(2,1)} \not= 0$ then $KO_2(C\sp*\sr(\Lambda, \gamma)) \cong KO_3(C\sp*\sr (\Lambda,\gamma)) \cong \Z^2_2$.

Now we will determine the maps $\eta$, $r$, and $c$, and also determine that $|KO_2 (C\sp*\sr (\Lambda,  \gamma))| = |KO_3(C\sp*\sr (\Lambda,  \gamma))| = 4$ 
(and hence that $d^2_{(2,1)} \neq 0$).
We start with the following segments of the long exact sequence \eqref{eq:CR-LES} relating $KO_*(C\sp*\sr (\Lambda,  \gamma))$ and $KU_*(C\sp* \sr(\Lambda,  \gamma))$:
\begin{align*} &KO_0(C\sp*\sr (\Lambda,  \gamma)) \xrightarrow{\eta_0} 
KO_1(C\sp*\sr (\Lambda,  \gamma)) \xrightarrow{c_1}
KU_1(C\sp*\sr (\Lambda,  \gamma)) \\
\text{and~~}
&KU_4(C\sp*\sr (\Lambda,  \gamma)) \xrightarrow{r_4}
KO_4(C\sp*\sr (\Lambda,  \gamma)) \xrightarrow{\eta_4} 
KO_5(C\sp*\sr (\Lambda,  \gamma)) \; .
\end{align*}
Since $|KO_1(C\sp*\sr (\Lambda,  \gamma))| = 4$ and $KO_0(C\sp*\sr (\Lambda,  \gamma)) \cong KU_2(C\sp* (\Lambda,  \gamma)) \cong \Z_2$ it follows that $\eta_0$ must be injective and $c_1$ must be surjective. Similarly $r_4$ must be injective and $\eta_4$ must be surjective.

Since $\eta_0$ is injective,  $r_0 = 0$ and $c_2$ is surjective. Since $c_1$ is surjective,  $r_7$ must be 0.
Continuing to use the long exact sequence \eqref{eq:CR-LES}, 
we find that 
\[ 
r_7 = 0, \quad  
\eta_7 = 1, \quad c_0 = 0, \quad r_6 = 1, \quad \eta_6 = 0, \quad c_7 = 1, \quad r_5 = 0, \quad  \eta_5 = 1, \quad c_6 = 0.\]
Also, since $\eta_4$ is surjective, we know that $c_5 = 0$.

Now $\eta_7$ and $\eta_0$ are both injective. It follows that $\eta_1: KO_1(C\sp*\sr(\Lambda, \gamma)) \to KO_2(C\sp*\sr(\Lambda, \gamma))$ cannot be injective also, due to the relation $\eta^3 = 0$.
So $| \ker \eta_1 |$ is either equal to 2 or to 4. But $| \ker \eta_1 | = | \im r_1 |$ and the latter cannot be equal to 4, since $KU_1(C\sp*\sr (\Lambda, \gamma) = \Z_2$. Thus $| \ker \eta_1 | = 2$.
Then the exact sequence
\[ 0 \rightarrow KO_1(C\sp*\sr (\Lambda,  \gamma))/\ker \eta_1 \xrightarrow{\eta_1} KO_2(C\sp*\sr(\Lambda,  \gamma)) \xrightarrow{c_2} KU_2(C\sp*\sr (\Lambda,  \gamma)) \rightarrow 0 \; \]
implies that $|KO_2(C\sp*\sr (\Lambda,  \gamma))| = 4$ (since the groups on the left and the right each have order 2). 
Thus $r_1$ is an injection and $c_3 = 0$. 
Moreover, the fact that $|KO_2(C\sp*\sr(\Lambda, \gamma))| \not= 8$ implies that $d^2_{(2,1)}\not= 0$ and consequently   $KO_2(C\sp*\sr(\Lambda, \gamma)) \cong \Z_2^2 \cong KO_3(C\sp*\sr(\Lambda, \gamma))$.  As $c_3 = 0$, we must have $\eta_2: KO_2(C\sp*\sr(\Lambda, \gamma)) \to  KO_3(C\sp*\sr(\Lambda, \gamma))$ onto, which implies that $\eta_2$ is an isomorphism and consequently $r_2=0$ and $c_4$ is onto.

As $\eta_2$ and $\eta_4$ are both surjective, the relation $\eta^3 = 0$  implies that $\eta_3:\Z_2^2 \to KO_4(C\sp*\sr(\Lambda, \gamma))$ cannot be surjective.
So $| \im \eta_3 |$ is equal to 0 or to 2. But if $| \im \eta_3 | = 0$, then $c_4: KO_4(C\sp*\sr(\Lambda, \gamma)) \to KU_4(C\sp*\sr(\Lambda,\gamma)) \cong \Z_2$ would be injective, which is not possible. 
Therefore $| \im \eta_3 | = | \ker c_4 | = 2$.

Now, we have calculated all of the groups $KO_*(C\sp*\sr (\Lambda,  \gamma))$, at least up to order, and the action of the maps $\eta, r, c$.  This enables us to compute the core of the $\CR$-module $K\crr(C\sp* \sr(\Lambda,  \gamma))$.

Recall from \eqref{eq:natural-transformations} that $\psi: KU_*(C\sp*\sr(\Lambda, \gamma))  \to KU_*(C\sp*\sr(\Lambda, \gamma)) $ satisfies $\psi^2 = 1$.  As $KU_i(C\sp*\sr(\Lambda, \gamma)) \cong \Z_2$ for each $i$, we may conclude that $\psi = 1$ for all $i$ in this case.  
It follows that $MU_i = (\ker(1 - \psi_i))/(\text{image}(1+\psi_i)) = KU_i(C\sp*\sr(\Lambda, \gamma))  \cong \Z_2$ for all $i$.

Furthermore, our descriptions of the maps $\eta_i$ above reveal that the groups $MO_i = \im \eta_{i-1}$ are as follows.
\[ \begin{array}{|c|c|c|c|c|c|c|c|c|c|}  
\hline  \hline 
  & \makebox[1cm][c]{0} & \makebox[1cm][c]{1} & 
\makebox[1cm][c]{2} & \makebox[1cm][c]{3} 
& \makebox[1cm][c]{4} & \makebox[1cm][c]{5} 
& \makebox[1cm][c]{6} & \makebox[1cm][c]{7} \\
\hline  \hline
MO_i
& \Z_2 & \Z_2 & \Z_2 & \Z_2^2 & \Z_2 & \Z_2 & \Z_2 &  0  \\ 
\hline  
\hline
\end{array} \]

Now from the $K$-theory calculations in \cite[Section~5.1, Table~5]{boersema2002} and \cite[Section~11, Table~2]{brs}, we find the the core of $K\crr( \mathcal{O}\pr_3) \cong K\crr( \mathcal{E}_3)$ is given by

\[ \begin{array}{|c|c|c|c|c|c|c|c|c|c|}  
\hline  \hline 
  & \makebox[1cm][c]{0} & \makebox[1cm][c]{1} & 
\makebox[1cm][c]{2} & \makebox[1cm][c]{3} 
& \makebox[1cm][c]{4} & \makebox[1cm][c]{5} 
& \makebox[1cm][c]{6} & \makebox[1cm][c]{7} \\
\hline  \hline
MO_i
& 0 & \Z_2 & \Z_2 & \Z_2 & \Z_2 & 0 & 0 &  0  \\ \hline 
MU_i
& \Z_2 & 0 & \Z_2 & 0 & \Z_2 & 0 & \Z_2 & 0 \\
\hline  
\hline
\end{array} \]

By comparing cores, we conclude that
\begin{align*}
K\crr( C\sp*\sr (\Lambda, \gamma)) & \cong \Sigma K\crr( \mathcal{O}\pr_3) \oplus \Sigma^{-2} K\crr( \mathcal{O}\pr_3) \\
& \cong \Sigma^{-1} K\crr( \mathcal{E}_3) \oplus  \Sigma^4 K\crr( \mathcal{E}_3)  \; .
\end{align*} 
That is,
\[ \begin{array}{|c|c|c|c|c|c|c|c|c|c|}  
\hline  \hline 
& \makebox[1cm][c]{0} & \makebox[1cm][c]{1} & 
\makebox[1cm][c]{2} & \makebox[1cm][c]{3} 
& \makebox[1cm][c]{4} & \makebox[1cm][c]{5} 
& \makebox[1cm][c]{6} & \makebox[1cm][c]{7} \\
\hline  \hline
KO_i(C\sp*\sr(\Lambda, \gamma))
& \Z_2 & \Z_4 & \Z_2^2  & \Z_2^2   & \Z_{4} & \Z_2  &  \Z_2 & \Z_2  \\
\hline  
KU_i(C\sp*\sr(\Lambda, \gamma))
& \Z_{2} & \Z_{2} & \Z_{2} & \Z_{2} & \Z_{2} & \Z_{2} & \Z_{2} & \Z_{2}  \\
\hline  
\hline
\end{array}\]

\end{proof}

\section{Questions}
\label{sec:questions}

Our investigation has highlighted many unanswered questions about higher-rank graph $C\sp*$-algebras and about the spectral sequence of Theorem \ref{sp-seq1} that computes $K\crr_*(C\sp*\sr(\Lambda, \gamma))$. First of all, Theorem \ref{sp-seq1} gives no information about the differential $d^r$ of the spectral sequence. How is this map determined by the higher-rank graph with involution $(\Lambda, \gamma)$? Can we compute $d^r$ from the combinatorial data -- the adjacency matrices, the factorization rule, the involution -- of $(\Lambda, \gamma)$? 

A related question is to better understand the role that $\gamma$ plays in these constructions. In setting up the spectral sequence, it is important to know which vertices are fixed and which are not fixed by $\gamma$. But beyond that, the action of $\gamma$ on the edges 
does not seem to play a role (unless it plays a role in determining the differential maps $d^r$ in a way that we are not aware of -- see the previous paragraph).

In fact, we know from \cite[Theorem 2.4]{boersema-MJM} that in the case of a 1-graph the isomorphism class of the real $C\sp*$-algebra $C\sp*\sr(\Lambda, \gamma)$ may depend on the action of $\gamma$ on the vertices of $\Lambda$ but not on the way $\gamma$ acts on the edges of $\Lambda$. We found that the proof of this theorem does not extend in an obvious way to the case of $k$-graphs where $k \geq 2$, but on the other hand we have no counter-examples to the analogous statement. How does $\gamma$ affect the $K$-theory of $C\sp*\sr (\Lambda, \gamma)$ and indeed how does $\gamma$ affect the isomorphism class of $C\sp*\sr(\Lambda, \gamma)$?

Another question concerns the functoriality of the spectral sequence. Specifically, suppose that $(\Lambda, \gamma)$ is a rank-$k$ graph with involution. Then for any $0 \leq \ell \leq k$, there is an obvious rank-$\ell$ graph $(\Lambda', \gamma')$ with involution:
\[ \Lambda' = \{ \lambda \in \Lambda: d(\lambda \in \N^\ell  = \{ (x_1, \dots, x_\ell, 0, \dots, 0) \mid x_i \in \N \} \subseteq \N^k \}\; .\] 
We define $\gamma'$ to be the restriction of $\gamma$.
There is an obvious corresponding map $i \colon C\sp*\sr(\Lambda', \gamma') \rightarrow C\sp*\sr(\Lambda, \gamma)$, which induces   a map  on $K$-theory:
\[ i_* \colon K\crr(C\sp*\sr(\Lambda', \gamma')) \rightarrow K\crr(C\sp*\sr(\Lambda, \gamma)) \; . \]
On the purely algebraic level, there is consequently  a homomorphism
from the chain complex associated to $C\sp*\sr(\Lambda', \gamma')$ to that associated to $C\sp*\sr(\Lambda, \gamma)$ that commutes with the differentials $\partial$. We conjecture that this map on the level of the chain complexes induces a map on the level of spectral sequences which commutes with the differentials $d^r$ and that it converges in the appropriate sense to the map $i_*$ on $K$-theory.

In particular, taking $\ell = 0$, this conjecture will provide a way to identify the class of any projection $[p_v]$ in $KO_0(C\sp*\sr(\Lambda, \gamma) )$ when $v$ is a vertex in $\Lambda$ fixed by $\gamma$, or the class of $[p_v + p_{\gamma(v)}]$ when $v$ is not fixed by $\gamma$. It would also provide a way to identify the class of the identity in $KO_0(C\sp*\sr(\Lambda, \gamma) )$ when $\Lambda$ is finite, which is part of the Elliot invariant when $C\sp*\sr(\Lambda, \gamma) )$ is simple and purely infinite. Such a result would be a direct generalization of Theorem~4.5 of \cite{boersema-MJM} and Theorem~3.2 of \cite{raeburn-szyman}. 

Finally, we wonder if our spectral sequence can be used to characterize the $\CR$-modules that can arise as $K\crr(C\sp*\sr(\Lambda, \gamma))$ where $(\Lambda, \gamma)$ is a rank $k$-graph. Corollary~4.3 of  \cite{boersema-MJM}  gives a necessary condition for a given $\CR$-module to be isomorphic to $K\crr(C\sp*\sr(\Lambda, \gamma))$, but we do not have a complete characterization, even when $\Lambda$ is a rank-$1$ graph. Which real Kirchberg algebras can be realized as $C\sp*\sr(\Lambda, \gamma)$ for some directed graph with involution $(\Lambda, \gamma)$? More generally, which real Kirchberg algebras can be realized as $C\sp*\sr(\Lambda, \gamma)$ for some higher rank graph with involution $(\Lambda, \gamma)$? In particular, the original question that motived this work is still unanswered: can we find concrete representations of the exotic real Cuntz algebras $\mathcal{E}_n$ using a family of higher-rank graphs with involution?

\bibliographystyle{amsalpha}
\bibliography{eagbib}
\end{document}